\documentclass[12pt]{amsart}
\usepackage{amscd}

\usepackage[utf8]{inputenc}
\usepackage[T1]{fontenc}
\usepackage{lmodern, enumerate}

\usepackage{amssymb,amsxtra}

\usepackage{nicefrac,mathtools}

\usepackage{microtype}
\usepackage[pdftitle={...},
 pdfauthor={Alcides Buss, Siegfried Echterhoff, Rufus Willett},
 pdfsubject={Mathematics}
]{hyperref}
\usepackage[lite]{amsrefs}

\usepackage{verbatim}

\usepackage{amsfonts}

\usepackage{amsthm}

\usepackage{amsmath}

\usepackage{mathabx}

\usepackage{enumerate}

\usepackage[all]{xy}

\usepackage{graphicx}

\usepackage{hyperref}

\newcommand{\N}{\mathbb{N}}

\newcommand{\C}{\mathbb{C}}

\newcommand{\h}{\mathcal{H}}

\newcommand{\B}{\mathcal{B}}

\newcommand{\Manoa}{M\=anoa}

\newcommand{\Hawaii}{Hawai\kern.05em`\kern.05em\relax i}

\newcommand{\alg}{\text{alg}}

\DeclareMathOperator{\Aut}{Aut}

\DeclareMathOperator{\Ind}{Ind}

\DeclareMathOperator{\supp}{\mathrm{supp}}

\renewcommand{\top}{{\operatorname{top}}}

\newcommand*{\nb}{\nobreakdash}
\newcommand*{\Star}{\(^*\)\nobreakdash-}

\newcommand{\ev}{\operatorname{ev}}

\newcommand*{\K}{\mathcal K}
\newcommand*{\M}{\mathcal{M}} 

\newcommand*{\cont}{C}
\newcommand*{\contz}{\cont_0}
\newcommand*{\contb}{\cont_b}
\newcommand*{\contub}{\cont_{ub}}
\newcommand*{\id}{\textup{id}}
\newcommand*{\Ad}{\textup{Ad}}

\newcommand*{\U}{\mathcal U}
\newcommand*{\E}{\mathcal E}

\newcommand*{\defeq}{\mathrel{\vcentcolon=}}

\newcommand*{\cstar}{\texorpdfstring{$C^*$\nobreakdash-\hspace{0pt}}{*-}}

\newcommand*{\onto}{\twoheadrightarrow}

\renewcommand*{\max}{\mathrm{max}}

\newcommand{\eps}{\varepsilon}

\newcommand*{\Cor}{\mathfrak{Corr}}

\newcommand{\as}{\operatorname{as}}

\newcommand{\emu}{{\mathcal{E}(\mu)}}

\newcommand{\free}{\scalebox{1.5}{$\ast$}}

\theoremstyle{plain}
\newtheorem{theorem}{Theorem}[section]
\newtheorem{lemma}[theorem]{Lemma}
\newtheorem{corollary}[theorem]{Corollary}
\newtheorem{proposition}[theorem]{Proposition}

\newtheorem{definition-theorem}[theorem]{Definition / Theorem}

\newtheorem*{conjecture*}{Conjecture}
\newtheorem*{theorem*}{Theorem}

\theoremstyle{definition}
\newtheorem{definition}[theorem]{Definition}
\newtheorem{example}[theorem]{Example}

\newtheorem{question}[theorem]{Question}

\theoremstyle{remark}
\newtheorem{remark}[theorem]{Remark}

\newtheorem*{example*}{Example}  
\newtheorem*{remark*}{Remark}

\begin{document}
\title{The minimal exact crossed product}

\author{Alcides Buss}
\email{alcides@mtm.ufsc.br}
\address{Departamento de Matem\'atica\\
 Universidade Federal de Santa Catarina\\
 88.040-900 Florian\'opolis-SC\\
 Brazil}

\author{Siegfried Echterhoff}
\email{echters@math.uni-muenster.de}
\address{Mathematisches Institut\\
 Westf\"alische Wilhelms-Universit\"at M\"un\-ster\\
 Einsteinstr.\ 62\\
 48149 M\"unster\\
 Germany}

\author{Rufus Willett}
\email{rufus@math.hawaii.edu}
\address{Mathematics Department\\
 University of \Hawaii~at \Manoa\\
Keller 401A \\
2565 McCarthy Mall \\
 Honolulu\\
 HI 96822\\
USA}

\begin{abstract}
Given a locally compact group $G$, we study the smallest exact crossed-product functor 
$(A,G,\alpha)\mapsto A\rtimes_\E G$ on the category of $G$-$C^*$-dynamical systems.
As an outcome, we show that the smallest exact crossed-product functor is automatically 
Morita compatible, and hence coincides with the functor $\rtimes_\E$ as introduced 
by Baum, Guentner, and Willett in their reformulation of the Baum-Connes conjecture
(see \cite{Baum:2013kx}). We show that the corresponding group algebra $C_\E^*(G)$
always coincides with the reduced group algebra, thus showing that the 
new formulation of the Baum-Connes conjecture coincides with the classical one 
in the case of trivial coefficients.

\end{abstract}


\maketitle

 \noindent
{\bf Erratum:} \emph{After publication of this manuscript, some gaps have unfortunately been found affecting some parts of the paper. We therefore included an appendix with an erratum at the end of this paper explaining the mistakes and keeping the original published version unchanged. }
 
\medskip

\section{Introduction}

The construction of crossed products $(A,G,\alpha)\mapsto A\rtimes_\alpha G$ 
provides a major source of examples in $C^*$-algebra theory and plays an important r\^ole in many applications
of $C^*$-algebras in other fields of mathematics, such as group representation theory and topology.
Classically, there were two crossed products attached to a given action 
$\alpha:G\to \Aut(A)$ of a locally compact group $G$ on a $C^*$-algebra $A$:
the maximal crossed product $A\rtimes_{\alpha,\max}G$, which is universal 
for covariant representations $(\pi, u)$ of the underlying dynamical system $(A,G,\alpha)$,
and the reduced crossed product $A\rtimes_{\alpha,r}G$, which can be defined as the image 
of $A\rtimes_{\alpha,\max}G$ under the regular covariant representation of the system.
Both crossed products are completions of the algebraic crossed product
$A\rtimes_\alg G=C_c(G,A)$ by $C^*$-norms $\|\cdot\|_\max$ and $\|\cdot\|_r$, respectively
and the fact that the identity map on $C_c(G,A)$  induces a 
quotient from $A\rtimes_\max G$ onto $A\rtimes_rG$ means that 
$\|f\|_\max\geq \|f\|_r$ for all $ f\in C_c(G,A)$.

More recently, the study of exotic crossed-product functors $(A,G,\alpha)\to A\rtimes_{\alpha,\mu}G$
came into the focus of research.  Here $A\rtimes_{\alpha,\mu}G$ is a completion of $C_c(G,A)$
with respect to a $C^*$-norm $\|\cdot\|_\mu$ satisfying
$$\|f\|_\max\geq \|f\|_\mu\geq \|f\|_r$$
for all $f\in C_c(G,A)$. The identity on $C_c(G,A)$ then induces surjective $*$-homomorphisms
$$A\rtimes_\max G\onto A\rtimes_\mu G\onto A\rtimes_r G$$
for every $G$-algebra $A$. 

The interest in exotic crossed products is 
motivated in a good part by the  failure of the classical Baum-Connes
conjecture, which predicted that a certain assembly map 
\begin{equation}\label{eq-as} \as^r: K_*^{\top}(G;A)\to K_*(A\rtimes_rG)
\end{equation}
for the $K$-theory of the reduced crossed product should always be an isomorphism.
However, it was shown by Higson, Lafforgue, and Skandalis in \cite{Higson:2002la} that 
the conjecture fails for certain groups discovered by Gromov \cite{Gromov:2003gf} (see \cite{Osajda:2014ys} for a concrete 
construction).  This failure is due to the fact that these groups are not exact in the sense that 
the sequence of reduced crossed products
\begin{equation}\label{exact group intro} 
\xymatrix{ 0 \ar[r] & I\rtimes_r G\ar[r] &  A\rtimes_rG \ar[r] &  (A/I)\rtimes_rG \ar[r] &  0}
\end{equation}
for a $G$-invariant ideal $I$ of $A$ may fail to be exact in general, even in a way that is detectable by $K$-theory.  This led to the idea
that one should replace the reduced crossed product by the smallest exact
  crossed-product functor which is compatible with Morita equivalences (at least in some weak sense --
 see Section \ref{sec-mor} below for the precise definition).
 Indeed, it has been shown in \cite{Baum:2013kx} that for every locally compact group
 $G$ a smallest exact Morita compatible  functor $A\mapsto A\rtimes_\E G$ always exists; moreover, if we replace 
 the reduced crossed product by $A\rtimes_\E G$ in (\ref{eq-as}) getting a new assembly map
  \begin{equation}\label{eq-asE} \as^\E: K_*^{\top}(G;A)\to K_*(A\rtimes_\E G),
\end{equation}
then the known counterexamples for the Baum-Connes conjecture disappear, some counterexamples become confirming examples, and the known confirming examples remain as such.   Note that for exact groups, i.e.,  groups for which (\ref{exact group intro}) is always exact, we have $A\rtimes_\E G=A\rtimes_rG$, and hence the new conjecture
coincides with the old one for those groups.

 The smallest exact Morita compatible crossed-product functor $\rtimes_\E$ has been studied further
 in \cite{Buss:2014aa, Buss:2015ty}, where it has been shown (among other things), that 
 for second countable $G$, its restriction to the category of separable $G$\nb-$C^*$-algebras
  enjoys other good functorial properties: It is functorial 
 for $G$\nb-equivariant correspondences and it allows a descent in equivariant $KK$-theory.
 On the other hand, in many respects our understanding of the functor $\rtimes_\E$ has been
very limited. Important questions are (among others):
\begin{itemize}
\item[{\bf Q1.}] What is the group algebra $C_\E^*(G):=\C\rtimes_\E G$? Do we always have 
$C_\E^*(G)= C_r^*(G)$, the reduced group algebra?
\item[{\bf Q2.}] Is the smallest exact {\em Morita compatible} crossed-product functor $\rtimes_\E$
identical to the smallest exact crossed-product functor?
\item[{\bf Q3.}] Can we give  concrete descriptions or constructions of the functor $\rtimes_\E$? 
\item[{\bf Q4.}] How can we relate the smallest exact Morita compatible 
functor $\rtimes_{\E_G}$ for a group $G$ to the same functor $\rtimes_{\E_H}$ for
a closed subgroup $H$ of $G$?
\end{itemize}
Note that a positive answer to Question Q1 would imply that the new Baum-Connes conjecture coincides with the classical one
in the case of the trivial coefficient algebra $A=\C$, 
which would  fit with the fact that so far there are no known counterexamples for the classical 
Baum-Connes conjecture in this case.

In this paper we will give positive answers to Questions Q1 and Q2 and give at least partial answers to 
Questions Q3 and Q4.
Given any fixed crossed-product functor $\rtimes_\mu$  for a group $G$
(which will be the reduced crossed-product functor 
in our main applications), we start in 
Section \ref{sec-halfexact} with the construction of a crossed-product functor $\rtimes_\emu$ which
is the smallest {\em half-exact} crossed-product functor (i.e., the analogue of sequence (\ref{exact group})  is exact 
at the middle term) that dominates $\rtimes_\mu$ in the sense that
$\|f\|_\emu\geq \|f\|_\mu$ for all $f\in C_c(G,A)$. 
 We  show (see Proposition \ref{group alg}) that the corresponding group algebra 
 $C_\emu^*(G)=\C\rtimes_\emu G$ always coincides with the group algebra $C_\mu^*(G):=\C\rtimes_\mu G$.
 In particular, $C_{\E(r)}^*(G)=C_r^*(G)$.
 
In  Section \ref{ex sec}, building on ideas developed around 
 Archbold's and Batty's  property C (see \cite{Archbold:1980aa} and the treatment of this property in \cite[Chapter 9]{Brown:2008qy})
 and work of Matsumura \cite{Matsumura:2012aa}, we prove
 
 \begin{theorem}[see Theorem \ref{thm-exact}]
 Let $\rtimes_\mu$ be a crossed-product functor for the locally compact group $G$. Then the following 
 are equivalent:
 \begin{enumerate}
 \item $\rtimes_\mu$ is half-exact;
 \item for every $G$-algebra $A$ there is a canonical $*$-homomorphism 
 $$A^{**}_c\rtimes_\mu G\to (A\rtimes_\mu G)^{**}$$
where $A_c^{**}$ denotes the $G$-continuous part of the double dual $A^{**}$ of $A$;
\item $\rtimes_\mu$ is exact.
\end{enumerate}
\end{theorem}
\noindent This theorem not only gives a new characterization of exact groups (when applied to the reduced crossed-product functor),
it also shows that the functor $\rtimes_{\E(r)}$ of Section \ref{sec-halfexact} is indeed the smallest exact exotic crossed-product functor.
In Section \ref{sec-mor} we then show that $\rtimes_{\E(r)}$  is Morita compatible,
which gives a positive answer to Question Q2 (i.e., $\rtimes_{\E(r)}=\rtimes_\E$), and, since  $C_\E^*(G)=C_{\E(r)}^*(G)=C_r^*(G)$,  also to 
Question Q1.

In Section \ref{sec-lift} we study certain equivariant lifting  properties for $G$-algebras, which give 
rise to more concrete descriptions of the smallest exact functor $\rtimes_\E$. 
In particular, we say that a $G$-algebra $C$ has the {\em weak equivariant lifting property} (WELP), if
 for any diagram of equivariant maps of the form  
$$
\xymatrix{ & B\ar[d]^-\pi\\ C \ar[r]^-\sigma \ar@{-->}[ur]^-{\widetilde{\sigma}} & B/J}
$$
with $\sigma$ a $*$-homomorphism and $\pi$ a quotient map,  the dashed arrow can always be filled in with a $G$-equivariant
ccp map $\tilde\sigma$. If $G$ is discrete, it is not difficult to see that for any $G$-algebra $A$, there always exists a short exact 
sequence $0\to I\to C\stackrel{\pi}{\to} A\to 0$ with $C$ satisfying (WELP), and 
it then follows that 
$$A\rtimes_{\E} G=\frac{C\rtimes_r G}{I\rtimes_rG}.$$
If $G$ is discrete and $A$ is unital, then $C$ can always be chosen to be the maximal group algebra $C^*(F_{N\times G})$ of a free group 
generated by a set $N\times G$ with $G$-action induced by the translation action on the second factor 
 on the generating set $N\times G$.
It follows in particular that for all $G$-algebras satisfying (WELP) we have $C\rtimes_{\E} G=C\rtimes_rG$. 
These include all equivariantly projective $G$-algebras $C$ as defined in \cite{Phillips:2015aa}.

Finally, in Section \ref{sec-BC} we show that for an {\em open} subgroup $H$ of $G$, the minimal exact functor $\rtimes_{\E_H}$
is always the {\em restriction} (see Section \ref{sec-BC} for the definition)
of the minimal exact functor $\rtimes_{\E_G}$ for $G$. This implies in particular, that validity of 
the reformulated Baum-Connes conjecture for a locally compact group $G$ passes to all open subgroups 
of $G$. Note that an analogue of this result for  {\em closed normal} subgroups of $G$ has been obtained in \cite{Buss:2015ty}.
\medskip

\noindent
{\bf Conventions:} The phrase \emph{$G$-algebra} will always mean a $C^*$-algebra equipped with a continuous action by $*$-automorphisms of a locally compact group $G$.  If $A$ is a $C^*$-algebra equipped with a not-necessarily-continuous action
$$
\alpha:G\to \text{Aut}(A)
$$
of a locally compact group $G$, then the \emph{continuous part} of $A$ is defined to be
$$
A_c:=\{a\in A\mid g\mapsto \alpha_g(a) \text{ is norm continuous}\},
$$
and is a $G$-algebra with the restricted action.  For a locally compact group $G$ and a $G$-algebra $A$, $C_c(G,A)$ denotes the collection of all compactly supported continuous functions from $G$ to $A$, equipped with the usual $*$-algebra operations.  The reduced and maximal completions of $C_c(G,A)$ will be denoted $A\rtimes_r G$ and $A\rtimes_{\max} G$.   A general crossed product functor as in \cite[Definition 3.2]{Buss:2015ty} will be denoted $\rtimes_\mu$, and the associated completion of $C_c(G,A)$ by $A\rtimes_\mu G$.  For most of the paper, there should be no ambiguity about which particular action a given $C^*$-algebra is equipped with; as such, we will not label crossed products with the name of the action unless it seems necessary to avoid confusion.  Finally, if $g\in G$, then we denote the canonical associated unitary in the multiplier algebra $\M(A\rtimes_\mu G)$ of a crossed product by $\delta_g$.

\medskip

\noindent
{\bf Acknowledgments:} This work was prompted by a suggestion of Narutaka Ozawa to consider norms induced by quotient maps from $G$-algebras of the form $C^*(F_{N\times G})$ as discussed above.  We are very grateful to Professor Ozawa for this initial suggestion.  The authors would also like to thank Erik Guentner and Hannes Thiel for useful conversations on some of the issues in this paper.  

Most of the work on this paper was carried out during a visit of the first and third authors to the second author at the Westf\"{a}lische Wilhelms-Universit\"{a}t M\"{u}nster; these authors would like to thank the second author, and that institution, for their hospitality.  

The authors were supported by Deutsche Forschungsgemeinschaft (SFB 878, Groups, Geometry \& Actions), by CNPq/CAPES -- Brazil, and by the US NSF (DMS 1401126 and DMS 1564281).

\section{Half exact crossed products}\label{sec-halfexact}

Throughout this section, $G$ denotes a locally compact group, and $\rtimes_\mu$ a fixed crossed-product functor for $G$ as in \cite[Definition 3.2]{Buss:2015ty}.  At this point, we do not assume that $\rtimes_\mu$ has any other properties beyond those in this basic definition; however we will need to specialize to crossed product functors satisfying more stringent conditions later, and will make clear when this comes up.
Our goal is to define a new crossed-product functor $\rtimes_{\emu}$, which should be thought of as the `best exact approximation to $\rtimes_\mu$', and indeed in Section \ref{ex sec} we will prove that it is the smallest exact crossed product that is larger than $\rtimes_\mu$.   We will spend most of this section proving some basic properties of $\rtimes_\emu$.

The reader unfamiliar with exotic crossed products is encouraged to just assume that $\rtimes_\mu=\rtimes_r$ throughout, which is certainly the most important special case.  Nonetheless, it seemed worthwhile to work in general as this causes no extra difficulties, and as it clarifies the `formal' nature of the constructions and proofs; by `formal' we mean that they rely on general $C^*$-algebra theory and functorial properties of $\rtimes_\mu$, and have nothing to do with the specific construction underlying the definition of $\rtimes_r$.

We first need some ancillary notation.

\begin{definition}\label{complete in}
Let $B\subseteq A$ be an equivariant inclusion of $C^*$-algebras.  Then $B\rtimes_{\mu,A}G$ denotes the completion of $C_c(G,B)$ for the norm it inherits as a subalgebra of $A\rtimes_\mu G$.
\end{definition}

Here is the main definition of this section.

\begin{definition}\label{quot norm}
Let $A$ be a $G$-algebra, and let 
$$
\xymatrix{ 0 \ar[r] & I \ar[r] & C \ar[r]^-\pi & A \ar[r] & 0 }
$$
be an equivariant short exact sequence.  Then we get a short exact sequence 
$$
\xymatrix{ 0 \ar[r] & I\rtimes_{\mu,C} G \ar[r] & C\rtimes_\mu G \ar[r] & \frac{C\rtimes_\mu G}{I\rtimes_{\mu,C} G} \ar[r] & 0 }.
$$
This gives rise to a (dense) $*$-algebra inclusion 
$$
C_c(G,A) \hookrightarrow \frac{C\rtimes_\mu G}{I\rtimes_{\mu,C} G}.
$$
The \emph{$\pi$-norm}\footnote{Of course, the $\pi$-norm also depends on the fixed crossed product $\rtimes_\mu$, but the `parent' crossed product should always be clear from context, so we do not include it in the notation.} on $C_c(G,A)$, denoted $\|\cdot\|_{\pi}$, is the norm induced by the above inclusion, and the corresponding completion is denoted $A\rtimes_{\pi} G$.  

The \emph{$\emu$-norm} on $C_c(G,A)$ is defined by 
$$
\|a\|_{\emu}:=\sup\{\|a\|_\pi\mid \pi:C\to A \text{ an equivariant surjection}\}
$$
and the corresponding completion of $C_c(G,A)$ is denoted $A\rtimes_\emu G$.
\end{definition}

Note that the supremum defining the $\emu$-norm is over a non-empty set: indeed, it contains the $\pi$-norm associated to the identity function $\pi:A\to A$.   Moreover, the supremum is finite as if we have an equivariant short exact sequence
$$
\xymatrix{0\ar[r] & I \ar[r] & C\ar[r]^-\pi & A  \ar[r] & 0 }
$$
then exactness of the maximal crossed product gives rise to a quotient map  
$$
A\rtimes_{\max}G =\frac{C\rtimes_{\max}G}{I\rtimes_{\max} G} \to \frac{C\rtimes_{\mu}G}{I\rtimes_{\mu,C} G},
$$
whence $\|a\|_\pi\leq \|a\|_{\max}$ for all $a\in C_c(G,A)$.  On the other hand, functoriality
of  $\rtimes_\mu$ gives rise to a quotient map
$$
\frac{C\rtimes_\mu G}{I\rtimes_{\mu,C}G}\to A\rtimes_\mu G
$$
so that $\|a\|_\mu\leq \|a\|_\pi$ for all $a\in C_c(G,A)$.  
Moreover, every crossed-product norm is assumed to satisfy $\|a\|_r\leq \|a\|_\mu$, where $\|\cdot\|_r$ denotes the reduced norm.
Hence we get the inequalities
\begin{equation}\label{norm sandwich}
\|a\|_r\leq \|a\|_\mu\leq  \|a\|_\emu\leq \|a\|_{\max} \quad \text{for all} \quad a\in C_c(G,A).
\end{equation}

\begin{proposition}\label{quot func}
Let $\phi:A\to B$ be an equivariant $*$-homomorphism.  Then the integrated form 
$$
\phi\rtimes G:C_c(G,A)\to C_c(G,B);\quad a\mapsto \phi\circ a
$$
extends (uniquely) to a $*$-homomorphism $A\rtimes_{\emu} G \to B\rtimes_{\emu} G$.  In particular, $\rtimes_{\emu}$ is a crossed-product functor in the sense of  \cite[Definition 3.2]{Buss:2015ty}.
\end{proposition}

\begin{proof}
Let
$$
\xymatrix{ 0 \ar[r] & I \ar[r] & C \ar[r]^-\pi & B\ar[r] & 0 }
$$
be an arbitrary equivariant short exact sequence.  Let $P=\{(c,a)\in C\oplus A\mid \pi(c)=\phi(a)\}$ be the pullback over the diagram
$$
\xymatrix{ P \ar[r]^{\pi_A} \ar[d]^{\pi_C} & A \ar[d]^-\phi \\ C\ar[r]^-\pi & B}
$$ 
(the maps labeled $\pi_A$ and $\pi_B$ are the restrictions of the coordinate projections from $C\oplus A$ to $P$).  The direct sum $G$-action on $C\oplus A$ restricts to an action on $P$ and we  thus obtain a commutative diagram of equivariant short exact sequences
$$
\xymatrix{ 0 \ar[r] & J \ar[d]  \ar[r] & P \ar[d]^-{\pi_C} \ar[r]^-{\pi_A} & A\ar[r] \ar[d]^-\phi & 0  \\ 0 \ar[r] & I \ar[r] & C \ar[r]^-\pi & B\ar[r] & 0 },
$$
where the map $J\to I$, which is an isomorphism, exists by commutativity.  Taking crossed products thus induces a map
$$
\frac{P\rtimes_\mu G}{J\rtimes_{\mu,P} G} \to \frac{C\rtimes_\mu G}{I\rtimes_{\mu,C} G}
$$
that agrees with $\phi\rtimes G$ on $C_c(G,A)$.  This implies that for all $a\in C_c(G,A)$ we have
$$
\|\phi\rtimes G(a)\|_{\pi}\leq \|a\|_{\pi_A}\leq \|a\|_\emu.
$$
Taking the supremum over all such $\pi$ now gives that $\|\phi\rtimes G(a)\|_{\emu}\leq \|a\|_{\emu}$ and thus that $\phi\rtimes G$ extends as claimed.  That $\rtimes_\emu$ is a crossed-product functor follows from functoriality of algebraic descent 
$$
(\phi:A\to B) \quad \mapsto \quad (\phi\rtimes G :C_c(G,A)\to C_c(G,B))
$$
and the inequality in Equation \eqref{norm sandwich}.
\end{proof}

We now show that the supremum defining the $\emu$-norm  is always attained.

\begin{proposition}\label{norm real}
Let $G$ be a locally compact group, let $\rtimes_\mu$ be a crossed product for $G$, and let $A$ be a $G$-algebra.  Then there exists an equivariant quotient map $\pi:C\to A$ such that $\|a\|_\emu=\|a\|_\pi$ for all $a\in C_c(G,A)$. If $A$ is unital, then $C$ can be chosen to be unital as well.
\end{proposition}

\begin{proof}
Let $S$ be a set of equivariant quotient maps $\pi_s:C_s\to A$ such that for every $a\in C_c(G,A)$,
$$
\|a\|_\emu=\sup_{s\in S}\|a\|_{\pi_s}.
$$
Define 
$$
C_0:=\Big\{(c_s)\in \prod_{s\in S} C_s\mid \pi_s(c_s)=\pi_t(c_t) \text{ for all } s,t\in S\Big\},
$$
and let $C$ be the continuous part of $C_0$, i.e., for all $c=(c_s)_{s\in S}\in C$ the map 
$$
G\to C; \quad g\mapsto \gamma_g(c):=(\gamma^s_g( c_s))_{s\in S}
$$
is continuous, where $\gamma^s:G\to \Aut(C_s)$ denotes the action of $G$ on $C_s$. For any fixed $t\in S$ let $\sigma_t:C_0\to C_t$ denote the projection. We claim that its restriction to $C$ is surjective.
For each $f\in C_c(G)$ and $c=(c_s)_{s\in S}\in C_0$, we define 
$$f*c:=(f*c_s)_{s\in S}\quad\text{with } f*c_s:=\int_G f(g) \gamma^s_g(c_s)\,dg.$$
We claim first that $f*c\in C$ for all $c\in C_0$.
To see this, it suffices to show that if $(g_j)$ is a net converging to the identity in $G$, then 
$$
\lim_{j}\sup_s \|\gamma^s_{g_j}(f*c_s)-f*c_s\|=0.
$$
Note however that for any $s\in S$, if $\delta_g$ denotes the Dirac mass at $g$, then 
\begin{align*}
\|\gamma^s_{g_j}(f*c_s)-f*c_s\| & =\|(\delta_{g_j}*f)*c_s-f*c_s\|\leq \|\delta_{g_j}*f-f\|_{L^1(G)}\|c_s\|\\ & \leq \|\delta_{g_j}*f-f\|_{L^1(G)}\|c\|;
\end{align*}
as $f$ and $c=(c_s)_{s\in S}$ are fixed, this tends to zero as $j$ tends to infinity at a rate independent of $s$ as required. 
It follows that the image $\sigma_t(C)\subseteq C_t$ contains all elements of the form $\{f* c_t: f\in C_c(G), c_t\in C_t\}$; hence to show surjectivity of $\sigma_t$, it suffices to show that this set is dense in $C_t$.
To see this, let $\mathcal V$ be a neighbourhood base of $e\in G$ and for each $V\in \mathcal V$ let
$f_V\in C_c(G)$ be a positive symmetric function with $\supp f_V\subseteq V$ and $\int_G f_V(g) dg=1$. Then
$f_V*c_t$ converges in norm to $c_t$ for any $c_t\in C_t$, $t\in S$.  This proves the claim.

For any fixed $t\in S$, there is thus a surjective quotient map
$$
C\to A;\quad (c_s)_{s\in S}\mapsto \pi_t(c_t).
$$
The definition of $C\subseteq C_0$ implies that this map does not depend on the choice of $t$, so we just denote it $\pi$.

We now have that $C$ is a $G$-algebra equipped with an equivariant surjection $\pi:C\to A$, so it remains to show that $\|a\|_{\emu}=\|a\|_\pi$ for all $a\in C_c(G,A)$.  For $s\in S$, recall that $\sigma_s:C\to C_s$ denotes the coordinate projection.  Then we get an equivariant commutative diagram
$$
\xymatrix{ 0 \ar[r] & I \ar[d] \ar[r] & C \ar[d]^-{\sigma_s} \ar[r]^-\pi & A \ar@{=}[d]\ar[r] & 0  \\ 0 \ar[r] & I_s \ar[r] & C_s  \ar[r]^-{\pi_s} & A \ar[r] & 0 },
$$
where $I$ and $I_s$ are the kernels of $\pi$ and $\pi_s$, respectively.  This gives rise to a $*$-homomorphism
$$
\frac{C\rtimes_\mu G}{I\rtimes_{\mu,C}G}\to \frac{C_s\rtimes_\mu G}{I_s\rtimes_{\mu,C_s}G}
$$
that restricts to the identity on $C_c(G,A)$; as this $*$-homomorphism is contractive, this implies that $\|a\|_\pi\geq \|a\|_{\pi_s}$ for all $a\in C_c(G,A)$.  As $s$ was arbitrary, the choice of $S$ then gives that $\|a\|_\pi\geq \|a\|_\emu$.  By definition of $ \|a\|_\emu$, this implies equality.
Suppose now that $A$ is unital. Let
$$
\xymatrix{ 0 \ar[r] & I \ar[r] & C \ar[r]^-\pi \ar[r] & A \ar[r] & 0 }.
$$
be a short exact sequence such that $\|\cdot\|_\pi=\|\cdot\|_\emu$ on $C_c(G,A)$ as above.
Let $\widetilde{C}$ denote the unitization of $C$ (even if $C$ is already unital) with the extended $G$-action that necessarily fixes the unit.  Let $\widetilde{\pi}:\widetilde{C}\to A$ denote the unique (equivariant) unital extension of $\pi$ to $\widetilde{C}$, and let $J$ be the kernel of $\widetilde{\pi}$.  Noting that $I$ is an ideal in $J$ with quotient $\C$, and taking crossed products, we get a commutative diagram
\begin{equation}\label{3x3}
\xymatrix{ & 0 \ar[d]& 0\ar[d] & 0\ar[d] & \\ 0 \ar[r] & I\rtimes_{\mu,C} G \ar[d]\ar[r] & C\rtimes_\mu G \ar[r] \ar[d]& A\rtimes_\pi G \ar[r] \ar[d]& 0 \\ 
0 \ar[r] & J\rtimes_{\mu,\widetilde{C}} G \ar[d]\ar[r] & \widetilde{C}\rtimes_\mu G \ar[r] \ar[d]& A\rtimes_{\widetilde{\pi}} G \ar[r] \ar[d]& 0  \\
0 \ar[r] & \C\rtimes_\mu G \ar[d]\ar[r] & \C\rtimes_\mu G \ar[r] \ar[d]& 0 \ar[r] \ar[d]& 0 \\
& 0 & 0 & 0 & }.
\end{equation}
The middle column is exact since the unit inclusion  $\sigma:\C\to \widetilde{C}$, which is a $G$-homomorphism, induces a splitting homomorphism 
$\sigma\rtimes G:\C\rtimes_\mu G \to \widetilde{C}\rtimes_\mu G$.  Let now $E:\widetilde{C}\rtimes_\mu G \to \widetilde{C}\rtimes_\mu G$ be defined by 
$$
E:=\text{Id}-(\sigma\rtimes G)\circ (\delta\rtimes G),
$$  
where $\delta:\widetilde{C}\to \C$ denotes the canonical quotient map.  Then $E$ is a bounded linear idempotent operator with norm at most two.  
Note that the definition of $E$ only needs functoriality of $\rtimes_\mu$ for $*$-homomorphisms. Moreover, it is straightforward to check that $E$ restricts to a map $C_c(G,J)\to C_c(G,I)$, whence it takes $J\rtimes_{\mu,\widetilde{C}} G$ onto $I\rtimes _{\mu,C} G$ (it has closed range as it is an idempotent), and acts as the identity on $I\rtimes_{\mu,\widetilde{C}}G$.  It follows from a diagram chase that if $a\in J\rtimes_{\mu,\widetilde{C}} G$ goes to zero under the quotient map to $\C\rtimes_\mu G$, then $E(a)=a$, and thus that $a\in I\rtimes_{\mu,C} G$.  Hence the left hand vertical column is also exact. 

To complete the proof, note that we now have that the left two columns in diagram \eqref{3x3} above are exact, while the rows are all exact by definition.  It follows from a diagram chase that the map $A\rtimes_{\pi}G \to A\rtimes_{\widetilde{\pi}} G$ is an isomorphism, and thus that for any $a\in C_c(G,A)$ we have
$$
\|a\|_{\widetilde{\pi}}=\|a\|_{\pi}=\|a\|_\emu,
$$
and we are done.
\end{proof}

\begin{remark} Note that the algebra $C$ constructed above depends strongly on $A$ and we have no idea about 
its general structure as a $C^*$-algebra. For instance, it is not clear whether we can always find a $C$ with the property in the proposition that is $\sigma$-unital if we assume that $A$ is $\sigma$-unital.

For discrete $G$ we shall see in Section \ref{sec-lift} below that we 
can get much more concrete descriptions of algebras $C$ and surjective morphisms $\pi:C\to A$ which attain 
the norm $\|\cdot\|_\emu$ for $G$.
\end{remark}

Using the above result, we can compute $\C\rtimes_\emu G$.

\begin{proposition}\label{group alg}
The canonical quotient map $\C\rtimes_\emu G\to \C\rtimes_\mu G$ is an isomorphism.  
\end{proposition}

\begin{proof}
Lemma \ref{norm real} implies in particular that there is a unital $G$-algebra $C$ and a 
$G$-invariant character $\pi:C\to \C$ such that $\|a\|_\emu=\|a\|_\pi$ for all $a\in C_c(G)=C_c(G,\C)$.

Any unital equivariant surjection $\pi:C\to \C$ splits equivariantly by the unit inclusion $*$-homomorphism $\C\to C$, which implies that the induced sequence 
$$
\xymatrix{ 0 \ar[r] & I\rtimes_{\mu,C} G \ar[r] & C\rtimes_\mu G \ar[r] & \C\rtimes_\mu G \ar[r] & 0 }
$$
is exact, and thus that $\|\cdot\|_\emu=\|\cdot\|_\pi=\|\cdot\|_\mu$ on $C_c(G)$.  The result follows.
\end{proof}

Our next aim is to show that $\emu$ is always half-exact as in the next definition, and is in fact minimal amongst all half-exact crossed-product functors dominating $\mu$.

\begin{definition}\label{def-proj-crossed} 
A crossed-product functor $A\mapsto A\rtimes_\mu G$ is called {\em half-exact}
if for every short exact sequence
of $G$-algebras 
$$
\xymatrix{ 0 \ar[r] & I \ar[r] & A \ar[r]^-\rho &  B \ar[r] & 0 }
$$
the sequence
$$
\xymatrix{ 0 \ar[r] & I\rtimes_{\mu,A} G \ar[r] & A\rtimes_\mu G \ar[r]^-{\rho\rtimes G} &  B\rtimes_\mu G \ar[r] & 0 }
$$
is exact, where $I\rtimes_{\mu, A}G$ is as in Definition \ref{complete in}.
\end{definition}

\begin{remark}\label{half ex ex}
A half-exact functor is exact if and only if it has the \emph{ideal property} of \cite[Definition 3.2]{Buss:2014aa}: This means that if $I\subseteq A$ is a $G$-invariant ideal, then the induced map $I\rtimes_\mu G \to A\rtimes_\mu G$ is injective.
The image of this homomorphism is $I\rtimes_{\mu,A}G$ so that $I\rtimes_{\mu,A}G\cong I\rtimes_\mu G$ canonically in this case.
\end{remark}

\begin{proposition}\label{quot lem}
The functor $\rtimes_\emu$ is half-exact.
\end{proposition}

\begin{proof}
Fix a $G$-invariant ideal $I$ in a $G$-algebra $A$ and write $\sigma:A\to A/I$ for the quotient map. Let
$$
\xymatrix{ 0 \ar[r] & J \ar[r] & C \ar[r]^{\pi} & A \ar[r] & 0}.
$$
be any equivariant short exact sequence such that $\|a\|_\emu=\|a\|_\pi$ for all $a\in C_c(G,A)$, as exists 
by Proposition \ref{norm real}.
This fits into a commutative diagram 
$$
\xymatrix{ & 0 \ar[d]& 0 \ar[d]& 0 \ar[d] &  \\  0\ar[r] & J\ar[r] \ar[d] & J\ar[r]  \ar[d]& 0 \ar[d] \ar[r]&0 \\
0\ar[r] & \pi^{-1}(I) \ar[r]  \ar[d]^-{\pi_I}& C\ar[r]^-{\sigma\circ \pi}  \ar[d]^-\pi \ar[r] & A/I \ar[d] \ar[r]&0 \\
0\ar[r] & I \ar[r] \ar[d]& A\ar[r]^-\sigma  \ar[d]& A/I \ar[d] \ar[r]&0 \\ & 0 & 0 & 0 & }
$$
with all rows and columns exact, where $\pi_I$ denotes the restriction of $\pi$ to $\pi^{-1}(I)$.  Taking crossed products gives a commutative diagram
{\small $$
\xymatrix{ & 0 \ar[d]& 0 \ar[d]& 0 \ar[d] &  \\  0\ar[r] & J\rtimes_{\mu,C}G \ar[r] \ar[d] & J\rtimes_{\mu,C} G\ar[r]  \ar[d]& 0 \ar[d] \ar[r]&0 \\
0\ar[r] & \pi^{-1}(I)\rtimes_{\mu,C} G \ar[r]  \ar[d] & C\rtimes_\mu G \ar[r]  \ar[d] \ar[r] & (A/I)\rtimes_{\sigma\circ \pi} G \ar[d] \ar[r]&0 \\
0\ar[r] & \frac{\pi^{-1}(I)\rtimes_{\mu,C} G}{J\rtimes_{\mu,C} G} \ar[r] \ar[d]& A\rtimes_{\emu} G\ar[r]  \ar[d]& \frac{A\rtimes_{\emu} G}{(\pi^{-1}(I)\rtimes_{\mu,C} G)~/~(J\rtimes_{\mu,C} G)} \ar[d] \ar[r]&0 \\ & 0 & 0 & 0 & }
$$}
where the first two columns and first two rows (at least) are exact.  The canonical isomorphisms
\begin{align*}
(A/I)\rtimes_{\sigma\circ \pi} G & \cong \frac{C\rtimes_\mu G}{\pi^{-1}(I)\rtimes_{\mu,C} G} \cong \frac{(C\rtimes_\mu G)~/~(J\rtimes_{\mu,C} G)}{(\pi^{-1}(I)\rtimes_{\mu,C} G)~ /~ (J\rtimes_{\mu,C} G)} \\ & \cong \frac{A\rtimes_{\emu} G}{(\pi^{-1}(I)\rtimes_{\mu,C} G)~ /~ (J\rtimes_{\mu,C} G)}
\end{align*}
identify the bottom right term with $A\rtimes_{\sigma\circ \pi} G$, and a diagram chase shows that the map 
$$
\frac{\pi^{-1}(I)\rtimes_{\mu,C} G}{J\rtimes_{\mu,C} G}\to A\rtimes_{\emu} G
$$ 
is injective.   Thus in fact all the rows and columns in the above diagram are exact. 

Now, it follows that $\frac{\pi^{-1}(I)\rtimes_{\mu,C} G}{J\rtimes_{\mu,C} G}$ identifies with the completion $I\rtimes_{\emu,A}G$ of $C_c(G,I)$ inside $A\rtimes_{\emu} G$, and that we have a canonical identification
$$
\frac{A\rtimes_{\emu} G}{I\rtimes_{{\emu},A}G}=(A/I)\rtimes_{\sigma\circ \pi} G.
$$
From this, we see that for any $a\in C_c(G,A/I)$, 
$$
\|a\|_{(A\rtimes_{{\emu}} G)~ / ~(I\rtimes_{{{\emu}},A}G)} = \|a\|_{(A/I)\rtimes_{\sigma\circ \pi} G}\leq \|a\|_\emu,
$$
where the right hand inequality follows from the definition of the $\emu$-norm on $C_c(G,A/I)$. 

To see the opposite inequality observe that the fact
 that $\rtimes_\emu$ is a functor gives a quotient map
$$
\frac{A\rtimes_\emu G}{I\rtimes_{\emu,A}G}\to (A/I)\rtimes_\emu G,
$$
so we are done.
\end{proof}

\begin{proposition}\label{min ex}
The crossed product $\rtimes_\emu$ is minimal amongst all half-exact crossed products that dominate $\rtimes_\mu$.
\end{proposition}

\begin{proof}
It is shown in Proposition \ref{quot lem} that $\rtimes_\emu$ is half-exact, and it dominates $\rtimes_\mu$ by the inequalities in Equation \eqref{norm sandwich}.  

Let $\rtimes_\nu$ be any other half exact crossed-product functor that dominates $\rtimes_\mu$.  Let $A$ be a $G$-algebra, and let 
$$
\xymatrix{0\ar[r] & J \ar[r] & C \ar[r]^-{\pi} & A \ar[r] & 0 }
$$
be a short exact sequence of $G$-algebras as in Proposition \ref{norm real}.  We thus get a commutative diagram
$$
\xymatrix{0\ar[r] & J\rtimes_{\nu,C} G \ar[r] \ar[d] & C\rtimes_\nu G  \ar[r] \ar[d] & A\rtimes_\nu G \ar[r] \ar@{-->} [d]& 0  \\ 0\ar[r] & J\rtimes_{\mu,C} G \ar[r] & C\rtimes_\mu G  \ar[r] & A\rtimes_{{\emu}} G \ar[r] & 0}
$$
of short exact sequences.  As the rows are exact, the dashed arrow can be filled in with a (necessarily surjective) $*$-homomorphism
which extends the identity on $C_c(G,A)$, and we are done.
\end{proof}

\begin{remark}\label{cub rem}
It is shown in \cite[Theorem A]{Brodzki:2015kb} that a locally compact\footnote{The reference given assumes that $G$ is second countable for this result; however, Kang Li has pointed out to us that one can use the structure theory of locally compact groups to deduce the general case from this.} group $G$ is exact if and only if the sequence
\begin{equation*}\label{eq:can-sequence}
\xymatrix{0\ar[r] & \contz(G)\rtimes_{r} G \ar[r] & \contub(G)\rtimes_r G  \ar[r] & \cont(\partial G)\rtimes_r G \ar[r]& 0  }
\end{equation*}
is exact, where $\contub(G)$ is the \cstar{}algebra of bounded left uniformly continuous functions endowed with the left translation $G$-action, that is, the continuous part of the translation $G$-action on $\contb(G)=M(\contz(G))$, 
and $\cont(\partial G)\defeq \contub(G)/\contz(G)$ is the quotient $G$-algebra. A simple chase with the diagram
$$
\xymatrix{0\ar[r] & \contz(G)\rtimes_{\mathcal{E}(r)} G \ar[r] \ar[d] & \contub(G)\rtimes_{\mathcal{E}(r)} G  \ar[r] \ar[d] & \cont(\partial G)\rtimes_{\mathcal{E}(r)} G \ar[r] \ar[d]& 0  \\ 0\ar[r] & \contz(G)\rtimes_{r} G \ar[r] & \contub(G)\rtimes_r G  \ar[r] & \cont(\partial G)\rtimes_{{r}} G \ar[r] & 0}
$$
together with fact that $\contz(G)\rtimes_{\mathcal{E}(r)} G=\contz(G)\rtimes_r G$ shows that $G$ is exact iff $\cont(\partial G)\rtimes_{\mathcal{E}(r)}G=\cont(\partial G)\rtimes_r G$, that is, the right vertical arrow in the above diagram is an isomorphism. For a general crossed-product functor $\rtimes_\mu$, one can follow exactly the same idea and prove part of the analogous result: if $\cont(\partial G)\rtimes_\mu G=\cont(\partial G)\rtimes_\emu G$, then the sequence
\begin{equation*}
\xymatrix{0\ar[r] & \contz(G)\rtimes_{\mu} G \ar[r] & \contub(G)\rtimes_\mu G  \ar[r] & \cont(\partial G)\rtimes_\mu G \ar[r]& 0  }
\end{equation*}
is exact. It is, however, not clear whether the exactness of a general $\rtimes_\mu$ can be detected by this sequence alone.
\end{remark}

\section{The ideal property and exactness}\label{ex sec}

Throughout this section, $G$ again denotes a locally compact group, and $\rtimes_\mu$ a crossed product functor for $G$.  We will mainly be interested in the case that $\rtimes_\mu=\rtimes_{\mathcal{E}(r)}$ or $\rtimes_\mu=\rtimes_r$, and the reader is encouraged to bear those two cases in mind; nonetheless, working in general is no more difficult, and seemed more conceptual, so we do this.

Our goal is to prove a necessary and sufficient condition for $\rtimes_\mu$ to be exact.
As a consequence we shall see that every half-exact crossed-product functor in the sense of Definition \ref{def-proj-crossed} is exact and, in particular, 
that the functor $\rtimes_\emu$ as constructed in the previous section is always exact.  It follows from this and Proposition \ref{min ex} that if $\rtimes_\mu=\rtimes_r$, then $\rtimes_\emu$ is the minimal exact crossed-product functor.  

The formulation and proof of our main theorem (see Theorem \ref{thm-exact} below)
is inspired in part by work of of Matsumura \cite{Matsumura:2012aa}, which is in turn inspired by the equivalence between property $C'$ of Archbold and Batty \cite{Archbold:1980aa} and exactness of the minimal tensor product as discussed in \cite[Chapter 9]{Brown:2008qy}.

We need some conventions, which are set up as follows.

\begin{definition}\label{conv}
Let $A$ be a  $G$-algebra with action $\alpha:G\to\Aut(A)$. 
Let $\iota\rtimes u: A\rtimes_\mu G\to \B(H)$ be a 
faithful and nondegenerate representation of $A\rtimes_\mu G$ on some Hilbert space $H$
and let $A'':=\iota(A)''$ and $(A\rtimes_\mu G)''$ denote the double commutants of $A$ and $A\rtimes_\mu G$ respectively in $\B(H)$.
We write $\iota'':A''\to (A\rtimes_\mu G)''$ for the canonical inclusion; note that this is the unique normal extension of $\iota$ to $A''$. 

Now let $I$ be a large directed set such that every element $a\in A''$ can be obtained as a limit of a bounded net 
$(a_i)_{i\in I}$  over the directed set $I$ in the strong* topology (for example, Kaplansky's density theorem implies that letting $I$ be a neighbourhood base of $0\in \B(H)$ for the strong* topology would work). Then 
$$
A^I:=\{(a_i)_{i\in I} \mid (a_i)_{i\in I} \text{ is a strong* convergent net}\}\subseteq \prod_{i\in I}A
$$
is a $C^*$-algebra, because multiplication and involution are strong*-continuous on bounded subsets of $\B(H)$.  There is moreover a $*$-homomorphism
$$
\rho:A^I\to A'';\quad  \rho\left((a_i)\right)=\text{strong*}\lim a_i\in A'',
$$
which is surjective by the choice of $I$.  Since $(\iota, u)$ is a covariant representation it follows that for each $g\in G$ the automorphism $\alpha_g$ of $A\cong \iota(A)$ extends to $A''$ via the automorphism $\alpha''_g:=\Ad u_g$ on $A''$.  Moreover, for each $g\in G$, we get a $*$-automorphism $\alpha^I_g$ of $\prod_{\in I}A$ defined componentwise by $\alpha^I_g((a_i)):=(\alpha_g(a_i))$, and the fact that the underlying unitary representation $u:G\to \mathcal{U}(H)$ is strong* continuous implies that $\alpha^I_g$ preserves $A^I$; we use the same notation for the restricted $*$-automorphism of $A^I$.  We thus get homomorphisms
\begin{equation}\label{actions}
\alpha'':G\to \text{Aut}(A''), \quad \alpha^I:G\to \text{Aut}(A^I),
\end{equation}
neither of which is necessarily continuous for the point-norm topologies on the right hand side.  The map $\rho$ is equivariant for these (not-necessarily-continuous) actions by strong* continuity of $u$ again.  Finally, we denote by $A^I_c$ and $A''_c$ the $C^*$-subalgebras of $A^I$ and $A''$ consisting of continuous elements for the actions in line \eqref{actions}, and note that $\rho$ restricts to an equivariant map
$$
\rho:A^I_c\to A''_c
$$
of $G$-algebras.
\end{definition}

\begin{lemma}\label{limit surj}
With notation as in Definition \ref{conv}, the map $\rho:A^I_c\to A''_c$ is surjective.
\end{lemma}

\begin{proof} We first claim that if $(a_i)_{i\in I}$ is a bounded net which converges to $a\in A''$ in the strong* topology, then for each compact subset $K\subseteq H$ and $\epsilon>0$ there exists an $i_0\in I$ such that
$$\forall \xi\in K,\forall i\geq i_0: \|(a_i-a)\xi\|, \|(a_i-a)^*\xi\|<\epsilon.$$
Indeed, let $\epsilon>0$ and let $R:=\sup_{i}\|a_i\|$. Then there exist finitely many vectors 
$\xi_1,\ldots, \xi_k\in H$ such that $K\subseteq \bigcup_{l=1}^k B_{\delta}(\xi_l)$ with $\delta=\frac{\epsilon}{4R}$.
Choose $i_0\in I$ such that for all $i\geq i_0$ and all $1\leq l\leq k$ we have
$$\| (a_i-a)\xi_l\|, \|(a_i-a)^*\xi_l\|< \frac{\epsilon}{2}.$$
Then for all $\xi\in K$ there exists $l\in \{1,\ldots,k\}$ such that $\|\xi-\xi_l\|\leq \frac{\epsilon}{4R}$ and then
for all $i\geq i_0$ we get
$$\|(a_i-a)\xi\|\leq \|a_i(\xi-\xi_l)\|+\|(a_i-a)\xi_l\|+\|a(\xi_l-\xi)\|<\epsilon$$
and similarly $\|(a_i-a)^*\xi\|<\epsilon$ for all $i\geq i_0$, completing the proof of the claim.

For $a\in A_c''$ and $f\in C_c(G)$ define
$$
f*a:=\int_G f(g)\alpha_g''(a)dg.
$$
As in the proof of Lemma \ref{norm real} it follows from the existence of an approximate identity for $L^1(G)$ in $C_c(G)$ that the collection 
$$\{f*a\mid f\in C_c(G),a\in A''_c\}$$ 
is norm-dense in $A''_c$. 
We now show that all such elements $f*a$ lie in the image of $\rho:A^I_c\to A_c''$, which will complete the proof.

Fix then $f\in C_c(G)$ and $a\in A''$.
Since the representation $u:G\to \U(H)$ is strong* continuous,
it follows that for all $\xi\in H$ the set $K_{\xi}:=\{u_g\xi: g\in \supp f\cup (\supp f)^{-1}\}$ is compact in $H$.
Hence, given an element 
$(a_i)_{i\in I}\in A^I$ with $a_i\stackrel{\text{strong*}}{\to} a$, then for all $g\in\supp(f)$
we have
\begin{align*}\| f(g)\alpha_g(a_i)\xi-f(g)\alpha_g''(a)\xi\| & =
\|f(g)u_g(a_i-a)u_g^*\xi\|\\ & \leq \|f\|_\infty \sup_{g\in\text{supp}(f)} \|(a_i-a)u_g^*\xi\|
\end{align*}
and by the claim at the start of the proof, the right hand side converges to zero, whence the left hand side converges to zero uniformly 
for $g$ in the support of $f$.  Thus it follows that for all $\xi\in H$ we have
$$(f*a_i)\xi=\int_G f(g)\alpha_g(a_i)\xi dg\to \int_G f(g) \alpha''_g(a)\xi dg= (f*a)\xi$$
and, similarly, $(f*a_i)^*\xi\to (f*a)^*\xi$.
Hence $\rho\left( (f*a_i)_i\right)=f*a$.
Now, as in the proof of Lemma \ref{norm real} we see that 
$(f*a_i)_i$ is a $G$-continuous element of $\prod_{i\in I}A$ (and hence of $A^I$), which completes the proof.  
\end{proof}

\begin{lemma}\label{prod cont}
With notation as in Definition \ref{conv}, the $*$-homomorphism
$$
(\iota''\circ \rho)\rtimes u:C_c(G,A^I_c)\to (A\rtimes_{\mu}G)''
$$
extends to a $*$-homomorphism
$$
A^I_c\rtimes_{\mu} G\to (A\rtimes_{\mu} G)''.
$$
\end{lemma}

\begin{proof}
It will suffice to show that if $f\in C_c(G,A^I_c)$, then 
$$
\|\iota''\circ \rho\circ f\|_{(A\rtimes_{\mu} G)''}\leq \|f\|_{A^I_c\rtimes_{\mu} G}.
$$
Write $f:G\to A^I_c\subseteq \prod_{i\in I}A$ as a net $(f_i)$ of functions $f_i:G\to A$; note that the net $(f_i)_i$ is equicontinuous, uniformly bounded, and all the $f_i$ have support in some fixed compact subset of $G$.  Computing, we get 
$$
\|\iota''\circ \rho\circ f\|_{(A\rtimes_{\mu} G)^{''}}=\Big\|\int_G\text{strong*-}\lim_{i}f_i(g)u_g dg\Big\|_{\mathcal{B}(H)}.
$$
Write 
\begin{equation}\label{limit fun}
f_\infty(g)=\text{strong*-}\lim_{i}f_i(g),
\end{equation}
so $f_\infty:G\to A''_c$ is a (norm) continuous and compactly supported function.  We first claim that 
$$
\int_G\text{strong*-}\lim_{i}f_i(g)u_g dg=\text{strong*-}\lim_{i}\int_Gf_i(g)u_g dg.
$$
Fix $\epsilon>0$.  Using uniform boundedness, uniform compact support, and equicontinuity of the net $(f_i)$, there is a finite subset $\{g_1,...,g_N\}$ of $G$ and scalars $\{t_1,...,t_N\}$ such that 
$$
\left\|\int_Gf_i(g){u_g} dg-\sum_{k=1}^Nt_kf_{i}(g_k){u_{g_k}}\right\|<\epsilon,
$$
for all $i$, and similarly for the limit function $f_\infty\in C_c(G,A''_c)$ as in line \eqref{limit fun} above we have
\begin{equation}\label{lim less}
\left\|\int_Gf_{\infty}(g){u_g} dg-\sum_{k=1}^Nt_kf_{\infty}(g_k){u_{g_k}}\right\|<\epsilon.
\end{equation}
Hence, using that strong* limits do not increase norms, 
$$
\left\|\text{strong*-}\lim_{i}\Big(\int_Gf_i(g){u_g} dg-\sum_{k=1}^Nt_kf_i(g_k){u_{g_k}}\Big)\right\|\leq \epsilon,
$$
and so using that strong* limits commute with finite linear combinations 
$$
\left\|\text{strong*-}\lim_{i}\int_Gf_i(g){u_g} dg-\sum_{k=1}^Nt_k f_\infty(g_k){u_{g_k}}\right\|\leq \epsilon.
$$
Combining this with the inequality in line \eqref{lim less}, we get that
$$
\left\|\int_Gf_{\infty}(g){u_g} dg-\text{strong*-}\lim_{i}\int_Gf_i(g){u_g} dg\right\|<2\epsilon~;
$$
as $\epsilon$ was arbitrary, this completes the proof of the claim.

Now, using the claim,
\begin{align}\label{strong lim}
\|\iota''\circ \rho\circ f\|_{(A\rtimes_{\mu}G)''}&=\Big\|\text{strong*-}\lim_{i}\int_Gf_i(g){u_g}dg\Big\|_{\mathcal{B}(H)}\nonumber\\
&\leq \sup_i\Big\|\int_Gf_i(g){u_g} dg\Big\|_{{\mathcal B(H)}},
\end{align}
where the inequality follows again as strong* limits do not increase norms.   On the other hand, the evaluations $\ev_i:A^I_c\to A$ induce norm-decreasing $*$-homomorphisms 
$\ev_i\rtimes G: A^I_c\rtimes_\mu G\to A\rtimes_\mu G$  for all $i\in I$, and therefore 
\begin{equation}\label{prod inq}
\|f\|_{A^I_c\rtimes_\mu G}\geq \sup_i\|f_i\|_{A\rtimes_\mu G}=\left\|\left(\int_Gf_i(g){u_g} dg\right)\right\|_{\prod_{i\in I} \mathcal B(H)},
\end{equation}
where the equality follows as the representation of $A\rtimes_\mu G$ into $\mathcal B(H)$ is faithful.  The norm on the right in line \eqref{strong lim} is the norm of the net $\left(\int_Gf_i(g){u_g} dg\right)_{i\in I}$ in the product $\prod_{i\in I} \mathcal B(H)$,   
whence combining line \eqref{strong lim} with line \eqref{prod inq} gives that 
$$
\| \iota\circ \rho\circ f\|_{(A\rtimes_\mu G)''}\leq \Big\|\left(\int_Gf_i(g){u_g} dg\right)\Big\|_{\prod_i \mathcal B(H)}\leq \|f\|_{A^I_c\rtimes_\mu G} 
$$
as desired.
\end{proof}

\begin{lemma}\label{mats lem}
Suppose that $\rtimes_\mu$ is a half-exact crossed-product functor in the sense of Definition \ref{def-proj-crossed}. Then
with notation as in Definition \ref{conv}, the $*$-homomorphism
$$
\iota''\rtimes u:C_c(G,A_c'')  \to (A\rtimes_{\mu} G)''
$$
extends to a $*$-homomorphism
$$
A_c''\rtimes_{\mu} G\to (A\rtimes_{\mu} G)''.
$$
\end{lemma}

\begin{proof}
Write $J$ for the kernel of the $*$-homomorphism $\rho:A^I_c\to A''_c$.  Lemma \ref{limit surj} says that $\rho$ is surjective, so we have a short exact sequence 
$$
\xymatrix{ 0 \ar[r] & J \ar[r] & A^I_c \ar[r]^-\rho & A''_c \ar[r] & 0 }.
$$
Note that the $*$-homomorphism $A^I_c\rtimes_{\mu} G\to (A\rtimes_{\mu} G)''$ of Lemma \ref{prod cont} 
contains $C_c(G,J)$ in its kernel, and so in the notation of Definition \ref{complete in} it induces a $*$-homomorphism
$$
\frac{A^I_c\rtimes_{\mu} G}{J\rtimes_{{\mu},A^I_c} G}\to (A\rtimes_{\mu} G)''.
$$
Since $\rtimes_\mu$ is half-exact,
this translates to a $*$-homomorphism 
$$
(A^I_c/J)\rtimes_{\mu} G\to (A\rtimes_{\mu} G)''
$$
and using the canonical isomorphism $A^I_c/J=A''_c$ this gives the desired homomorphism 
$$
A''_c\rtimes_{\mu} G\to (A\rtimes_{\mu} G)'',
$$
so we are done.
\end{proof}

Look now at the special case of Definition \ref{conv} where $\iota\rtimes u: A\rtimes_\mu G\to \B(H)$ is the universal 
representation of $A\rtimes_\mu G$.  Then $(A\rtimes_\mu G)''$ is the enveloping von Neumann algebra
of $A\rtimes_\mu G$, which identifies with the double dual $(A\rtimes_\mu G)^{**}$.

Let $\iota^{**}:A^{**}\to (A\rtimes_\mu G)^{**}$ denote 
the  normal extension of the representation $\iota:A\to (A\rtimes_\mu G)^{**}\subseteq \B(H)$, and abusing notation, use $\iota^{**}$ also for the restriction of this map to the continuous part $A^{**}_c$.
Then $(\iota^{**}, u)$ is a covariant representation of the $C^*$-dynamical system
$(A_c^{**}, G)$ into $(A\rtimes_\mu G)^{**}$ and therefore integrates to a $*$-homomorphism
\begin{equation}\label{eq-A**}
\iota^{**}\rtimes u: C_c(G, A^{**}_c)\to (A\rtimes_\mu G)^{**}.
\end{equation}
Notice that the map $\iota^{**}:A^{**}\to (A\rtimes_\mu G)^{**}$ is injective: indeed, if $H_A$ is the universal representation for $A$, then we have a sequence of canonical maps
$$
\xymatrix{ A^{**}_c \ar[r]^-{\iota^{**}} & (A\rtimes_\mu G)^{**} \ar[r] & (A\rtimes_r G)^{**} \ar[r] & \B(H_A\otimes L^2(G)) },
$$
where the last map is the normal extension of the regular representation associated to the universal representation of $A$, and whose composition is easily seen to be injective.  A similar reasoning shows that the integrated form \eqref{eq-A**} is injective, but we shall not use this fact.
Thus $\iota^{**}$ is a special case of the map $\iota''$ from Definition \ref{conv}.

\begin{theorem}\label{thm-exact}
Let $\rtimes_\mu$ be a crossed-product functor for the locally compact group $G$. Then the following are equivalent:
\begin{enumerate}
\item $\rtimes_\mu$ is half-exact;
\item for every $G$-algebra $A$ and every faithful representation $\iota\rtimes \mu:A\rtimes_\mu G\to \B(H)$, with notation as in Definition \ref{conv} we have a 
$*$-homomorphism 
$$\iota''\rtimes u: A_c''\rtimes_\mu G\to (A\rtimes_\mu G)'';$$
\item for every $G$-algebra $A$, the map of line (\ref{eq-A**}) extends to a $*$-homomorphism
$$\iota^{**}\rtimes u: A_c^{**}\rtimes_\mu G\to (A\rtimes_\mu G)^{**};$$
\item $\rtimes_\mu$ is exact.
\end{enumerate}
\end{theorem}

\begin{proof}
The implication (1) $\Rightarrow$ (2) follows from Lemma \ref{mats lem}, and (3) is a special case of (2).  

Suppose now that (3) holds and let $0\to J\to C\to A\to 0$ be any short exact sequence of $G$-algebras.
We need to show that this sequence descends to a short exact sequence
$$\xymatrix{ 0\ar[r] & J\rtimes_\mu G \ar[r] &  C\rtimes_\mu G\ar[r] & A\rtimes_\mu G\ar[r] &  0}.$$
It is clear that the map $C\rtimes_\mu G\to A\rtimes_\mu G$ is surjective. 
Consider the commutative diagram
\begin{equation}\label{eq-diag}
\begin{CD}
0 @>>>  J\rtimes_\mu G  @>>> C\rtimes_\mu G  @>>> A\rtimes_\mu G @>>> 0\\
@. @VVV @VVV @VVV \\
0 @>>> J_c^{**}\rtimes_\mu G @>>> C_c^{**}\rtimes_\mu G  @>>> A_c^{**}\rtimes_\mu G @>>> 0\\
@. @VVV @VVV @VVV \\
0 @>>>  (J\rtimes_\mu G)^{**}  @>>> (C\rtimes_\mu G)^{**}  @>>> (A\rtimes_\mu G)^{**} @>>> 0
\end{CD}
\end{equation}
where the top three vertical arrows are induced by functoriality of $\rtimes_\mu$, and the bottom three vertical arrows are as in assumption (3); note that the vertical compositions are just the canonical inclusions of each algebra into its double dual.
Since the inclusion $J\rtimes_\mu G\to (J\rtimes_\mu G)^{**}$ is injective (and similarly for $C$ and $A$), it follows that 
the upper vertical arrows are all injective.
Since $C_c^{**}$ decomposes as the direct sum of the $G$-algebras $J_c^{**}$ and $(C/J)_c^{**}=A_c^{**}$, it follows 
that the map $J_c^{**}\rtimes_\mu G \to  C_c^{**}\rtimes_\mu G$ is split injective. Hence the upper left square 
$$
\begin{CD}
J\rtimes_\mu G  @>>> C\rtimes_\mu G\\
@VVV @VVV \\
 J_c^{**}\rtimes_\mu G @>>> C_c^{**}\rtimes_\mu G  
 \end{CD}
 $$
 of diagram (\ref{eq-diag}) implies injectivity of $J\rtimes_\mu G  \to  C\rtimes_\mu G$. 
 
 Suppose now that $x\in C\rtimes_\mu G$ goes to $0$ in $A\rtimes_\mu G$. 
 Then its image in $C_c^{**}\rtimes_\mu G$ is mapped to $0\in A_c^{**}\rtimes_\mu G$, and hence must lie in 
 $ J_c^{**}\rtimes_\mu G$ by exactness of the middle horizontal sequence. Therefore, 
 $x$  lies in the intersection $ (J\rtimes_\mu G)^{**} \cap C\rtimes_\mu G$ inside  $(C\rtimes_\mu G)^{**}$.
 By \cite[Lemma 9.2.6]{Brown:2008qy} this intersection equals $J\rtimes_\mu G$, and we are done.
 
 The implication (4) $\Rightarrow$ (1) is trivial.
\end{proof}

\begin{remark}\label{ab rem}
Property (3) in Theorem \ref{thm-exact} is a direct analogue of \emph{property $C'$} of Archbold and Batty \cite[Definition 2.2]{Brown:2008qy}.  Hence the equivalence of (3) and (4) is an analogue for crossed-product functors of the fact that property $C'$ for a $C^*$-algebra $B$ is equivalent to exactness of the functor $A\mapsto A\otimes B$ (that is, the exactness of $B$); see  \cite[Proposition 9.2.7]{Brown:2008qy} for a proof of this.  The proof of Theorem \ref{thm-exact} above is (indirectly) inspired by the proof of the cited proposition.  We were directly inspired by work of Matsumura \cite{Matsumura:2012aa}, who proved (4) implies (3) for the special case when $G$ is discrete and $\rtimes_\mu$ is the reduced crossed product.
\end{remark}

\begin{remark}\label{sier ex}
Inspection of the proof of Theorem \ref{thm-exact} shows that if the map $\iota^{**}\rtimes u:A^{**}_c\rtimes_r G \to (A\rtimes_r G)^{**}$ exists, then any short exact sequence of the form $0\to I\to A \to B\to 0$ descends to a short exact sequence
$$
\xymatrix{ 0 \ar[r] & I\rtimes_r G \ar[r] & A\rtimes_r G \ar[r] & B\rtimes_r G \ar[r] & 0 },
$$
or in other words, that $A$ is an exact $G$-algebra in the sense of  \cite[Definition 1.2]{Sierakowski:2010aa}.  In particular, if $A=C_{ub}(G)$, then using Remark \ref{cub rem}, existence of the map $\iota^{**}\rtimes u:C_{ub}(G)_c^{**}\rtimes_r G \to (C_{ub}(G)\rtimes_r G)^{**}$ is equivalent to exactness of $G$.
\end{remark}

As an immediate corollary of this discussion and Theorem \ref{thm-exact}, we get the following characterization of exact groups.

\begin{corollary}\label{exact group}
Let $G$ be a locally compact group. Then the following are equivalent:
\begin{enumerate}
\item for every $G$-algebra $A$ and every faithful representation $\iota\rtimes \mu:A\rtimes_r G\to \B(H)$, with notation as in Definition \ref{conv} we have a 
$*$-homomorphism 
$$\iota''\rtimes u: A_c''\rtimes_r G\to (A\rtimes_r G)'';$$
\item for every $G$-algebra $A$, the map of line (\ref{eq-A**}) extends to a $*$-homomorphism
$$\iota^{**}\rtimes u: A_c^{**}\rtimes_r G\to (A\rtimes_r G)^{**};$$
\item for the $G$-algebra $A=C_{ub}(G)$, the map of line (\ref{eq-A**}) extends to a $*$-homomorphism
$$\iota^{**}\rtimes u: C_{ub}(G)_c^{**}\rtimes_r G\to (C_{ub}(G)\rtimes_r G)^{**};$$
\item $G$ is exact. \qed
\end{enumerate}
\end{corollary}

Finally, we have the following immediate corollary of Theorem \ref{thm-exact}, Proposition \ref{quot lem}, and Proposition \ref{min ex}.

\begin{corollary}\label{min exact cor}
For a given crossed product $\rtimes_\mu$, $\rtimes_\emu$ is the minimal exact crossed-product functor that dominates $\rtimes_\mu$.  In particular, $\rtimes_{\mathcal{E}(r)}$ is the minimal exact crossed-product functor among all crossed-product functors for $G$.\qed
\end{corollary}

\section{Morita compatibility}\label{sec-mor}

Throughout this section, $G$ denotes a locally compact group, and $\rtimes_\mu$ a crossed product functor.  As before, the reader is encouraged to assume that $\rtimes_\mu=\rtimes_r$, which is the most important special case, but the general case causes no extra difficulties.

Our goal is to show that \emph{Morita compatibility} as defined in \cite[Definition 3.3]{Baum:2013kx} (see also Definition \ref{mor com u} below) passes from $\rtimes_\mu$ to $\rtimes_\emu$, as long as the input $\rtimes_\mu$ has the ideal property.  In particular, $\rtimes_{\mathcal{E}(r)}$ is Morita compatible.  From this, it follows readily that $\rtimes_{\mathcal{E}(r)}$ agrees with the \emph{minimal exact Morita compatible functor} $\rtimes_{\mathcal{E}}$  of \cite[Theorem 3.13]{Baum:2013kx} on the category of all $G$-algebras, and, if $G$ is second countable, with the \emph{minimal exact correspondence functor} $\rtimes_{\mathcal{E}_{\mathfrak{Corr}}}$ of \cite[Corollary 8.8]{Buss:2014aa} on the category of separable $G$-algebras (\cite[Corollary 8.13]{Buss:2014aa}).   

To state the definition of Morita compatibility, we need some notation.  Let $H$ be a Hilbert space equipped with a $G$-action $u$, and let $\Ad u$ denote the induced action by conjugation on the compact operators $\mathcal{K}=\mathcal{K}(H)$.  Let $A$ be a $G$-algebra, and equip $A\otimes \mathcal{K}$ with the tensor product action.  Consider the $*$-homomorphism defined on the level of pre-completed crossed products and algebraic tensor products by the formula 
$$
\Psi_{\alg}:C_{c}(G,A)\odot \mathcal{K}\to C_{c}(G,A\otimes \mathcal{K}),\quad \Psi(a\otimes k)(g):= a(g)\otimes ku_g^*.
$$
Completing to the maximal crossed products and spatial tensor product, we get a $*$-homomorphism
\begin{equation}\label{psi def}
\Psi_{\max}:(A\rtimes_{\max}G) \otimes \mathcal{K}\to (A\otimes \mathcal{K})\rtimes_{\max}G.
\end{equation}
which is well-known to be a $*$-isomorphism.  An explicit inverse to $\Psi_{\max}$ is constructed as follows.  Consider the $*$-homomorphism
$$
\pi:A\otimes \mathcal{K}\to \M((A\rtimes_{\max}G) \otimes \mathcal{K});\quad a\otimes k \mapsto {\iota(a)}\otimes k,
$$
where $\iota:A\to \M(A\rtimes_{\max}G)$ denotes the canonical inclusion,
and the unitary representation 
$$
v:G \to \M((A\rtimes_{\max}G) \otimes \mathcal{K});\quad g\mapsto \delta_g\otimes u_g.
$$
Then the pair $(\pi,v)$ is readily checked to be covariant.  The integrated form gives a $*$-homomorphism 
\begin{equation}\label{phi def}
\Phi_{\max}:(A\otimes \mathcal{K})\rtimes_{\max}G\to (A\rtimes_{\max}G) \otimes \mathcal{K},
\end{equation}
which one can check is the inverse to $\Psi_{\max}$.  

Now, let $\rtimes_\mu$ be an arbitrary crossed-product functor.  Then postcomposing $\Phi_{\max}$ and $\Psi_{\max}$ with the canonical quotient maps from maximal to $\mu$-crossed products gives $*$-homomorphisms
\begin{equation}\label{phi n psi}
\begin{split}
&\Psi_{\max,\mu}:(A\rtimes_{\max}G) \otimes \mathcal{K}\to (A\otimes \mathcal{K})\rtimes_{\mu}G  \\ & \Phi_{\max,\mu}:(A\otimes \mathcal{K})\rtimes_{\max}G\to (A\rtimes_{\mu}G)\otimes \mathcal{K}.
\end{split}
\end{equation}

\begin{definition}\label{mor com u}
Let $\rtimes_\mu$ be a crossed-product functor, and let $H$ be a Hilbert space equipped with a $G$-action $u$.  The functor $\rtimes_\mu$ is \emph{$u$-Morita compatible} if for any $G$-algebra $A$ the $*$-homomorphism 
$$
\Psi_{\max,\mu}:(A\rtimes_{\max}G) \otimes \mathcal{K}\to (A\otimes \mathcal{K})\rtimes_{\mu}G.
$$
from line \eqref{phi n psi} above descends to a $*$-isomorphism
$$
\Psi_\mu:(A\rtimes_{\mu}G) \otimes \mathcal{K}\to (A\otimes \mathcal{K})\rtimes_{\mu}G
$$
Following \cite[Definition 3.3]{Baum:2013kx}, $\rtimes_\mu$ is \emph{Morita compatible} if it is $u$-Morita compatible with $u$ the tensor product of the left regular and trivial representations on $L^2(G)\otimes \ell^2(\N)$.
\end{definition}

Note that the maximal and reduced crossed products are $u$-Morita compatible for any $u$.

\begin{remark}\label{mor com sep}
One can separate checking $u$-Morita compatibility into two questions as follows.
\begin{enumerate}[(i)]
\item Does $\Psi_{\max,\mu}$ descend to a $*$-homomorphism 
$$
\Psi_\mu:(A\rtimes_{\mu}G) \otimes \mathcal{K}\to (A\otimes \mathcal{K})\rtimes_{\mu}G~?
$$
\item Does $\Phi_{\max,\mu}$ descend to a $*$-homomorphism
$$
\Phi_\mu:(A\otimes \mathcal{K})\rtimes_{\mu}G\to (A\rtimes_{\mu}G) \otimes \mathcal{K}~?
$$  
\end{enumerate}
The crossed product $\rtimes_\mu$ is $u$-Morita compatible if and only if the answer to both of these questions is `yes', and in that case the descended $*$\nb-homomorphisms $\Phi_\mu$ and $\Psi_\mu$ will automatically be mutually inverse (as they are mutually inverse on dense subalgebras).  
\end{remark}

One can at least always answer question (i) positively in the presence of the ideal property: recall this means that $\rtimes_\mu$ takes an equivariant inclusion $I\subseteq A$ of an ideal to an injective map $I\rtimes_\mu G \to A\rtimes_\mu G$.

\begin{lemma}\label{always ext}
Let $H$ be a Hilbert space equipped with a $G$-action $u$, let $A$ be a $G$-algebra, and let $\rtimes_\mu$ be a crossed product with the ideal property.  Then the $*$-homomorphism 
$$
\Psi_{\max,\mu}:(A\rtimes_{\max}G) \otimes \mathcal{K}\to (A\otimes \mathcal{K})\rtimes_{\mu}G
$$
of line \eqref{phi n psi} above descends to a $*$-homomorphism
$$
\Psi_\mu:(A\rtimes_{\mu}G)\otimes \mathcal{K}\to (A\otimes \mathcal K)\rtimes_{\mu}G.
$$
\end{lemma}

\begin{proof}
We have an equivariant $*$-homomorphism 
$$
A\to \M((A\otimes \mathcal{K})\rtimes_\mu G);\quad a\mapsto a\otimes 1
$$
and a unitary representation 
$$
G\to \M((A\otimes \mathcal{K})\rtimes_\mu G); \quad g\mapsto \delta_g.
$$
These form a covariant pair for $(A,G)$, which integrates to a $*$-homomorphism
$$
C_c(G,A)\to \M((A\otimes \mathcal{K})\rtimes_\mu G).
$$
As $\rtimes_\mu$ has the ideal property, \cite[Lemma 3.3]{Buss:2014aa} implies that the integrated form extends to a $*$-homomorphism
$$
\Psi_A:A\rtimes_{\mu}G\to \M((A\otimes \mathcal{K})\rtimes_{\mu} G).
$$
On the other hand, we have a $*$-homomorphism 
$$
\mathcal{K}\to \M((A\rtimes_{\max}G) \otimes \mathcal{K});\quad k\mapsto 1\otimes k.
$$
Postcomposing this with the map induced on multipliers by $\Psi_{\max}$ gives a $*$-homomorphism
$$
\mathcal{K}\to \M((A\otimes \mathcal{K})\rtimes_{\max}G)
$$
and postcomposing again with the map on multipliers induced by the canonical quotient map $(A\otimes \mathcal{K})\rtimes_{\max}G\to (A\otimes \mathcal{K})\rtimes_{\mu}G$ gives a $*$-homomorphism
$$
\Psi_{\mathcal{K}}:\mathcal{K}\to \M((A\otimes \mathcal{K})\rtimes_{\mu}G).
$$
The images of $\Psi_A$ and $\Psi_{\mathcal{K}}$ commute, so they combine to give a $*$-homomorphism 
$$
\Psi_A\odot \Psi_{\mathcal{K}}:(A\rtimes_{\mu}G)\odot \mathcal{K}\to \M((A\otimes \mathcal{K})\rtimes_{\mu}G)
$$
on the algebraic tensor product.  Checking on generators, this $*$-homomorphism actually takes image in $(A\otimes \mathcal{K})\rtimes_{\mu}G$, not just in the multiplier algebra.  It moreover extends to the spatial tensor product by nuclearity of $\mathcal{K}$, giving us
$$
\Psi_\mu:(A\rtimes_{\mu}G)\otimes \mathcal{K}\to (A\otimes \mathcal{K})\rtimes_{\mu}G,
$$
and one checks on generators that this is exactly the desired map.
\end{proof}

\begin{proposition}\label{q mor com}
Fix a unitary $G$-representation $u$, and let $\rtimes_\mu$ be a $u$-Morita compatible crossed product with the ideal property.  Then the crossed product $\rtimes_\emu$ is $u$-Morita compatible. In particular, if $\rtimes_\mu$ has the ideal property and is Morita compatible, then so is $\rtimes_\emu$.
\end{proposition}

\begin{proof}
Let $A$ be a $G$-algebra. The crossed product $\rtimes_\emu$ is exact by Corollary \ref{min exact cor}, whence in particular has the ideal property.  Hence Lemma \ref{always ext} implies that we have a $*$-homomorphism
$$
\Psi_\emu:(A\rtimes_{\emu}G)\otimes \mathcal{K}\to (A\otimes \mathcal{K})\rtimes_{\emu}G.
$$
To complete the proof, it suffices as in Remark \ref{mor com sep} to show that the $*$-homomorphism $\Phi_{\max,\emu}$ descends to a $*$-homomorphism
$$
\Phi_\emu:(A\otimes \mathcal{K})\rtimes_{\emu}G\to (A\rtimes_{\emu}G)\otimes \mathcal{K}.
$$
Let 
$$
\xymatrix{ 0 \ar[r] & I \ar[r] & C \ar[r]^-\pi & A \ar[r] & 0 }
$$
be an equivariant short exact sequence such that $A\rtimes_\pi G=A\rtimes_\emu G$ (this exists by Proposition \ref{norm real}).  Tensoring the above sequence by the $G$-algebra $\mathcal{K}$ (equipped as always with the action $\Ad u$, with the tensor products given the tensor product action) and taking crossed products gives a short exact sequence 
$$
\xymatrix{ 0 \ar[r] & (I\otimes \mathcal{K})\rtimes_{\mu}G \ar[r] & (C\otimes \mathcal{K})\rtimes_{\mu}G \ar[r] & \frac{(C\otimes \mathcal{K})\rtimes_{\mu}G}{(I\otimes \mathcal{K})\rtimes_{\mu}G} \ar[r] & 0 };
$$ 
here we have used the ideal property for $\rtimes_\mu$ to identify the closure of $C_c(G,I\otimes\mathcal{K})$ inside $(C\otimes \mathcal{K})\rtimes_{\mu}G$ with $(I\otimes \mathcal{K})\rtimes_{\mu}G$.
The definition of $(A\otimes \mathcal{K})\rtimes_{\emu}G$ implies that the identity map on $C_{c}(G,A)$ extends to a (surjective) $*$-homomorphism
\begin{equation}\label{phi 1}
(A\otimes \mathcal{K})\rtimes_{\emu}G\to \frac{(C\otimes \mathcal{K})\rtimes_{\mu}G}{(I\otimes \mathcal{K})\rtimes_{\mu}G}.
\end{equation}
On the other hand, the assumption that $\rtimes_\mu$ is $u$-Morita compatible gives isomorphisms
$$
\Phi_\mu:(C\otimes \mathcal{K})\rtimes_{\mu}G\to (C\rtimes_{\mu}G)\otimes \mathcal{K} \quad \text{and} \quad \Phi_\mu:(I\otimes \mathcal{K})\rtimes_{\mu}G\to (I\rtimes_{\mu}G)\otimes \mathcal{K}, 
$$
which together with exactness of the functor $B\mapsto B\otimes \mathcal{K}$, the definition of the $\pi$-norm, and the choice of $\pi:C\to A$, give rise to a map
\begin{equation}\label{phi 2}
\frac{(C\otimes \mathcal{K})\rtimes_{\mu}G}{(I\otimes \mathcal{K})\rtimes_{\mu}G}\to \Big(\frac{C\rtimes_{\mu}G}{I\rtimes_{\mu}G}\Big)\otimes \mathcal{K}=(A\rtimes_\pi G)\otimes \mathcal{K}=(A\rtimes_\emu G)\otimes \mathcal{K}.
\end{equation}
Composing the maps from lines \eqref{phi 1} and \eqref{phi 2} thus gives a $*$-homomorphism
$$
(A\otimes \mathcal{K})\rtimes_{\emu}G\to (A\rtimes_{\emu}G)\otimes\mathcal{K};
$$
checking on generators, one sees that this is the desired $*$-homomorphism $\Phi_\emu$, so we are done.
\end{proof}

The next corollary is immediate from \cite[Proposition 8.10]{Buss:2014aa}.

\begin{corollary}\label{str mor com}
Suppose that $\rtimes_\mu$ is a Morita compatible crossed product functor with the ideal property.
Then the restriction of $\rtimes_\emu$ to the category of $\sigma$-unital $G$-algebras is strongly Morita compatible in the sense of \cite[Definition 4.7]{Buss:2014aa}. \qed
\end{corollary}

The following corollary is immediate from Proposition \ref{q mor com} and Corollary \ref{min exact cor} above.

\begin{corollary}\label{min ex mc}
The crossed product $\rtimes_{\mathcal{E}(r)}$ is the same as the minimal exact Morita compatible functor $\rtimes_\mathcal{E}$ of 
\cite[Theorem 3.13]{Baum:2013kx}. \qed
\end{corollary}

The following corollary is immediate from Corollary \ref{min ex mc} and \cite[Corollary 8.14]{Buss:2014aa}.

\begin{corollary}\label{corr fun}
Let $G$ be second countable. Then, on the category of separable $G$-algebras, the crossed-product functor $\rtimes_{\mathcal{E}(r)}$ agrees with the minimal exact correspondence functor $\rtimes_{\mathcal{E}_{\mathfrak{Corr}}}$ of \cite[Corollary 8.8]{Buss:2014aa}. \qed
\end{corollary}

Note that this answers \cite[Questions 8.5, (iv) and (v)]{Baum:2013kx}.  Indeed, question (iv) asks whether $\C\rtimes_{\mathcal{E}}G$ equals $C^*_r(G)$, and Proposition \ref{group alg} and Corollary \ref{min ex mc} imply that this is always true.  On the other hand, question (v) asks whether $\rtimes_{\mathcal{E}}$ can be a KLQ functor for non-exact $G$; as a KLQ functor is uniquely determined by what it does on $\C$ and as $\rtimes_r\neq \rtimes_{\mathcal{E}(r)}$ for a non-exact group $G$, the answer `always' to question (iv) shows that the answer to question (v) is `never'.

To finish this section, we discuss which group algebras (i.e.,\ $C^*$-algebra completions of $C_c(G)$) can appear as $\C\rtimes_\mu G$ for certain types of crossed product functors $\rtimes_\mu$.  As discussed in \cite[Section 2]{Buss:2015ty}, every group algebra can be viewed as the group \cstar{}algebra $C^*_E(G)$ associated to a $G$-invariant weak*-closed $G$-invariant subspace $E$ of the Fourier-Stietjes algebra $B(G)$ containing the Fourier algebra $A(G)$. 

The Brown-Guentner construction \cite{Brown:2011fk} (see also the discussion in \cite[Section 3]{Buss:2015ty}) shows that any such $C^*_E(G)$ arises as the group algebra associated to an exact crossed product $\rtimes_{E_{BG}}$.  On the other hand, the Kaliszewiski-Landstad-Quigg construction \cite{Kaliszewski:2012fr} (see also the discussion in \cite[Section 3]{Buss:2015ty} again) shows that if $E$ is an ideal in $B(G)$, then $C^*_E(G)$ arises as the group $C^*$-algebra of a functor $\rtimes_{E_{KLQ}}$ that is $u$-Morita compatible for any $u$, and that has the ideal property; conversely \cite[Corollary 5.7]{Buss:2014aa}\footnote{\cite[Corollary 5.7]{Buss:2014aa} is written for correspondence functors, but inspection of the proof shows that it holds in this slightly more general setting.} shows that if $C^*_E(G)$ arises as the group algebra of any functor with these properties, then $E$ must be an ideal.

However, Brown-Guentner crossed products are generally not Morita compatible, and Kaliszewiski-Landstad-Quigg crossed products are generally not exact (see the discussion in \cite[Section 4]{Buss:2015ty}).  Thus it was not previously clear which group algebras can arise as $\C\rtimes_\mu G$ for a functor that is both exact and Morita compatible; this is a natural question, as such functors seem to behave best with respect to the Baum-Connes assembly map.

\begin{corollary}\label{4.8}
Let $G$ be a locally compact group.   Then a group algebra $C^*_E(G)$ is of the form $\C\rtimes_\mu G$ for a functor $\rtimes_\mu$ that is exact and $u$-Morita compatible for any $u$ if and only if $E$ is an ideal in $B(G)$.
\end{corollary}
\begin{proof}
Take the Kaliszewiski-Landstad-Quigg functor $\rtimes_\mu$ associated to $E$ and consider $\rtimes_{\emu}$. 
This is an exact Morita compatible crossed-product functor by Theorems~\ref{thm-exact} and Proposition~\ref{q mor com}; and it has the same group algebra $C^*_E(G)$ as the original functor $\rtimes_\mu$ by Proposition~\ref{group alg}.  The converse follows from \cite[Corollary 5.7]{Buss:2014aa} again, noting that the conditions in the statement give what is needed for the proof of that corollary to work.
\end{proof}

\section{Lifting properties and attaining norms}\label{sec-lift}

Throughout this section, $G$ is a locally compact group (although we will need to assume that $G$ is discrete for some results), and $\rtimes_\mu$ an associated crossed product. 

Recall from Lemma \ref{norm real} that the $\emu$-norm is always attained via a fixed
$G$-equivariant surjection $\pi:C\to A$. Our goal in this section is to get further information about  
possible choices for such surjections, at least in the case when $G$ is discrete.
The following definition, due to Phillips, S\o{}rensen, and Thiel \cite{Phillips:2015aa}, will be very useful here.

\begin{definition}\label{ep}
A $G$-algebra $C$ is \emph{equivariantly projective} if whenever $B$ is a $G$-algebra, $J\subseteq B$ is a $G$-invariant ideal, and $\phi:C\to B/J$ is an equivariant $*$-homomorphism, there is an equivariant $*$-homomorphism $\widetilde{\phi}:C\to B$ that lifts $\phi$ (in other words, if $\pi:B\to B/J$ is the quotient map, then the diagram 
$$
\xymatrix{ & B\ar[d]^-\pi \\ C \ar[r]^-\phi \ar[ur]^-{\widetilde{\phi}} & B/J}
$$
commutes).

We say that $C$ is equivariantly projective in the unital category if it has the above property, but with all $C^*$-algebras and maps appearing above assumed unital.
\end{definition} 

\begin{proposition}\label{ep ind}
Let $\rtimes_\mu$ be a crossed product.  Let $A$ be a $G$-algebra, and let $C$ be an equivariantly projective $G$-algebra equipped with an equivariant quotient map $\pi:C\to A$.  Then for any $a\in C_c(G,A)$, $\|a\|_\pi=\|a\|_\emu$.

Moreover, if $A$ is unital, the same conclusion follows if we assume that $C$ is equivariantly projective in the unital category and equipped with a unital quotient map $\pi:C\to A$. 
\end{proposition}

\begin{proof}
We need to show that if $\sigma:B\to A$ is any other equivariant quotient map, then for any $a\in C_c(G,A)$, $\|a\|_\sigma\leq \|a\|_\pi$.  
Consider the diagram
$$
\xymatrix{ & B\ar[d]^-\sigma \\ C \ar[r]^-\pi \ar@{-->}[ur]^-{\widetilde{\pi}} & A.}
$$
Equivariant projectivity for $C$ implies that the dashed arrow can be filled in by an equivariant $*$-homomorphism $\widetilde{\pi}:C\to B$.  Letting $I$ and $J$ be the kernels of $\pi$ and $\sigma$ respectively, we get a commutative diagram
\begin{equation}\label{lift ses}
\xymatrix{ 0 \ar[r] & I \ar[r] & C \ar[r]^-{\pi} \ar[d]^-{\widetilde{\pi}} & A \ar[r] \ar@{=}[d] & 0 \\ 0 \ar[r] & J \ar[r] & B \ar[r]^-{\sigma} & A \ar[r] & 0.}
\end{equation}
Commutativity of this diagram gives that $\widetilde{\pi}$ restricts to a map from $I$ to $J$.  Hence we get a commutative diagram
$$
\xymatrix{ 0 \ar[r] & I\rtimes_{\mu,C} G \ar[r] \ar[d]^-{\widetilde{\pi}|_I\rtimes_r G} & C\rtimes_\mu G \ar[r] \ar[d]^-{\widetilde{\pi}\rtimes_r G} & \frac{C\rtimes_\mu G}{I\rtimes_{\mu,C} G} \ar[r] & 0 \\ 0 \ar[r] & J\rtimes_{\mu,B} G \ar[r] & B\rtimes_\mu G \ar[r] & \frac{B\rtimes_\mu G}{J\rtimes_{\mu,B}G} \ar[r] & 0. }
$$
Commutativity implies that $\widetilde{\pi}\rtimes_r G$ induces a $*$-homomorphism on quotients
$$
\frac{C\rtimes_\mu G}{I\rtimes_{\mu,C} G}\to \frac{B\rtimes_\mu G}{J\rtimes_{\mu,B}G}, 
$$
which by commutativity of the diagram in line \eqref{lift ses} restricts to the identity on the $*$-subalgebra $C_c(G,A)$ of both sides.  This implies that 
$$
\|a\|_\sigma\leq \|a\|_\pi
$$
for any $a\in C_c(G,A)$ and we are done in the general case.

The statement in the unital case follows from essentially the same argument as by Lemma \ref{norm real} we may assume $B$ to be unital as well.
\end{proof}

As it gives some interesting examples, we will also explore a weakening of equivariant projectivity.  To explain the terminology, recall (see, e.g.\ \cite[Definition 13.1.1]{Brown:2008qy}) that a $C^*$-algebra $C$ has the \emph{lifting property} (LP) if whenever $\phi:C\to B/J$ is a contractive completely positive (ccp) map into a quotient $C^*$-algebra, there exists a ccp lift $\psi:C\to B$.  

\begin{definition}\label{weak lp}
A $G$-algebra $C$ has the \emph{weak equivariant lifting property} (WELP) if whenever $B$ is a $G$-algebra, $J\subseteq B$ is a $G$-invariant ideal, and $\phi:C\to B/J$ is an equivariant $*$-homomorphism, then there is an equivariant ccp map $\widetilde{\phi}:C\to B$ that lifts $\phi$ (in other words, if $\pi:B\to B/J$ is the quotient map, then the diagram 
$$
\xymatrix{ & B\ar[d]^-\pi \\ C \ar[r]^-\phi \ar[ur]^-{\widetilde{\phi}} & B/J}
$$
commutes).

We say that $C$ has the \emph{unital weak equivariant lifting property} (UWELP) if it has the above property, but with all maps and $C^*$-algebras assumed unital.
\end{definition}

\begin{remark}\label{elp rem}
It would be natural to define a stronger \emph{equivariant lifting property} (ELP): by analogy with the LP, one would here ask for equivariant ccp lifts of equivariant ccp maps $\phi$ as in the above, rather than just for equivariant $*$-homomorphic $\phi$.  We do not know if any non-trivial examples of ELP $G$-algebras exist for non-compact $G$, however.  This seems an interesting question.

Even for the trivial group the WELP is, a priori, weaker than the ordinary LP, although both are equivalent for separable \cstar{}algebras by an application of Stinespring's dilation theorem as in the proof of \cite[Theorem 13.1.3]{Brown:2008qy}.
\end{remark}

The proof of the following result is the same as that of Proposition \ref{ep ind} -- with equivariant ccp maps replacing equivariant $*$-homomorphisms at appropriate points -- and thus omitted.

\begin{proposition}\label{welp ind}
Let $\rtimes_\mu$ be a crossed product which is functorial for completely positive maps\footnote{For example, the reduced crossed product.  See \cite[Theorem 4.9]{Buss:2014aa} for some equivalent conditions.}.  Let $A$ be a $G$-algebra, and let $C$ be a $G$-algebra with the WELP, and equipped with an equivariant quotient map $\pi:C\to A$.  Then for any $a\in C_c(G,A)$, $\|a\|_\pi=\|a\|_\emu$. 

Moreover, if $A$ is unital, the same conclusion follows if we assume that $C$ has the UWELP and is equipped with a unital quotient map $\pi:C\to A$. \qed
\end{proposition}

The following corollary gives a way to compute $\rtimes_\emu$ in some special cases.  Note that $\C$ is clearly equivariantly projective in the unital category, so this gives a slightly different approach to proving Proposition \ref{group alg}.

\begin{corollary}\label{lp=red}
If $\rtimes_\mu$ is a crossed-product functor, and if $A$ is equivariantly projective, or equivariantly projective in the unital category, then $A\rtimes_\mu G=A\rtimes_\emu G$.  

If $\rtimes_\mu$ is a crossed-product functor which is functorial for completely positive maps, and if $A$ has either the WELP or the UWELP, then $A\rtimes_\mu G=A\rtimes_\emu G$.
\end{corollary}

\begin{proof}
In either case, Proposition \ref{ep ind} or Proposition \ref{welp ind} implies that the $\emu$-norm on $C_c(G,A)$ equals the $\pi$-norm where $\pi:A\to A$ is the identity map; clearly this is just the $\rtimes_\mu$-norm, however.
\end{proof}

Having got through the above, it is maybe not clear that interesting equivariantly projective, or even WELP, $G$-algebras exist.  Moreover, to get much use out of the above results, we would need to show that for any $G$-algebra $A$ there exists an equivariant surjection $\pi:C\to A$, where $C$ has the WELP.  Unfortunately, we can only prove this in the discrete case, and must leave the general locally compact case as a question for now.

The following result is essentially a special case of \cite[Proposition 2.4]{Phillips:2015aa}; we nonetheless give a direct proof for the reader's convenience.

\begin{proposition}\label{ep exist}
Let $G$ be a discrete group, let $X$ be a set, and let $C$ be the universal $C^*$-algebra generated by a set $\{c_{x,g}\mid (x,g)\in X\times G\}$ of positive contractions indexed by the set $X\times G$.  Equip $C$ with the $G$-action induced by the set action 
$$
g:(x,h)\mapsto (x,gh)
$$
and universality.  Then $C$ is equivariantly projective, and admits a surjective equivariant $*$-homomorphism onto any $G$-algebra generated by a set of positive contractions of cardinality at most that of $X$.
\end{proposition}

\begin{proof}
We first show that $C$ is equivariantly projective.  Let then $\phi:C\to B/J$ be an equivariant $*$-homomorphism.  For each $x\in X$, choose a positive contraction $b_x\in B$ that lifts $\phi(c_{x,e})$. 
Write $\beta$ for the action of $G$ on $B$, and let $\widetilde{\phi}:C\to B$ be the $*$-homomorphism uniquely defined by the map
$$
c_{x,g}\mapsto \beta_g(b_x)
$$
on generators.  Clearly this is equivariant and lifts $\phi$ on the generators; as it is a $*$-homomorphism it is thus equivariant and lifts $\phi$ on all of $C$.

Now, let $A$ have a generating set $S$ of positive contractions of cardinality at most that of $X$.  Write $\alpha$ for the action of $G$ on $A$.  Choose a surjective map $f:X\to S$.  Let now $\pi:C\to A$ be the $*$-homomorphism uniquely determined by the map
$$
c_{x,g}\mapsto \alpha_g(f(x)).
$$
This is equivariant on generators, so everywhere, and is surjective as $S$ generates $A$.  
\end{proof}

Another interesting example (and the one that originally inspired this work) is as follows.

\begin{proposition}\label{free welp}
Let $G$ be a discrete group, let $X$ be a set, let $F_{X\times G}$ be the free group on $X\times G$, and let $C:=C^*_{\max}(F_{X\times G})$ be the maximal group $C^*$-algebra of $F_{X\times G}$.  Equip $C$ with the $G$-action induced by the set action 
$$
g:(x,h)\mapsto (x,gh)
$$
and universality.  Then $C$ has the UWELP, and admits a surjective unital equivariant $*$-homomorphism onto any unital $G$-algebra generated by a set of unitaries of cardinality at most that of $X$.
\end{proposition}

\begin{proof}
The statement about the existence of a quotient map $C\to A$ follows from essentially the same construction as in the proof of Proposition \ref{ep exist}: we leave the details to the reader.  It remains to show that $C$ has the UWELP, so let $\phi:C\to B/J$ be a unital equivariant $*$-homomorphism.  For each $(x,g)\in X\times G$, let $u_{x,g}$ be the corresponding generating unitary for $C$, and choose a contractive lift $b_{x}\in B$ of $\phi(u_{x,e})\in B/J$.  Now define 
$$
v_x:=\begin{pmatrix} b_x & (1-b_xb_x^*)^{1/2} \\ (1-b_x^*b_x)^{1/2} & -b_x^* \end{pmatrix}\in M_2(B),
$$
which is unitary.  Let $\beta$ denote the action on $B$.  Universality implies that the map defined on generators by 
$$
u_{x,g}\mapsto \beta_g(v_x)
$$
extends to an equivariant $*$-homomorphism $C\to M_2(B)$.  The top left corner of this $*$-homomorphism is the desired ucp equivariant lift of $\phi$.  
\end{proof}

\begin{remark}
It is proved in \cite{Phillips:2015aa}*{Proposition~2.4} that the Bernoulli shift $G$-actions on free products of the form $A=\free_{g\in G}B$ are always $G$\nb-equivariantly projective if the base \cstar{}algebra $B$ is (non-equivariantly) projective. This result contains Proposition~\ref{ep exist} as a special case by taking $B\cong\contz(0,1]$, the universal \cstar{}algebra generated by a positive contraction, which is a projective \cstar{}algebra. Similarly, one can show that $A=\free_{g\in G}B$ has the WELP for the Bernoulli shift $G$-action provided that $B$ has the WELP for the trivial group action. An analogous version of the UWELP holds for unital free products (i.e., amalgamated over $\C$), generalising Proposition~\ref{free welp}. To prove these assertions one can use Boca's result from \cite[Theorem 3.1]{Boca:1991aa} on free products of ccp (or ucp) maps. Indeed, in the non-equivariant situation, this idea has already been used by Boca to prove that certain lifting properties are preserved by free products in \cite{Boca:1997aa}.
\end{remark}

\section{Restriction}\label{sec-BC}

Suppose that $H$ is a closed subgroup of a locally compact group $G$ and that $\rtimes_\mu$ is a crossed-product functor for $G$.  Our goal in this section is to study the relationship between the minimal exact crossed products for $H$ and $G$, with a view to applications to the (reformulated) Baum-Connes conjecture.

For an $H$-algebra $(A,\alpha)$  consider the induced $G$\nb-algebra\linebreak $(\Ind_H^G(A,\alpha), \Ind\alpha)$
in which
$$\Ind_H^G(A,\alpha):=\left\{F\in C_b(G,A): \begin{matrix} \alpha_h(F(gh))=F(g)\;\forall g\in G, h\in H,\\
\text{and}\; (gH\mapsto \|F(g)\|)\in C_0(G/H)\end{matrix}\right\}.$$
The $G$\nb-action on $\Ind_H^G(A,\alpha)$ is given by $\big(\Ind\alpha_g(F)\big)(k)=F(g^{-1}k)$.
Green's imprimitivity theorem (see \cite[Theorem 17]{Green:1978aa} or \cite[Section 2.6]{Echterhoff:2009jo}) provides a
natural equivalence bimodule $X(A,\alpha)$ between\linebreak $\Ind_H^G(A,\alpha)\rtimes_{\Ind\alpha,\max}G$ and $A\rtimes_{\alpha,\max}H$.
Let
$$I_{\mu}=\ker\Big(\Ind_H^G(A,\alpha)\rtimes_{\max}G\to \Ind_H^G(A,\alpha)\rtimes_{\mu}G\Big).$$
By the Rieffel correspondence between ideals in
$\Ind_H^G(A,\alpha)\rtimes_{\max}G$ and ideals in $A\rtimes_{\max}H$ there is a unique ideal $I_{\mu|_H}\subseteq A\rtimes_{\max}H$ such that
$X(A,\alpha)$ factors through an equivalence bimodule
$X_\mu(A,\alpha)$ between $\Ind_H^G(A,\alpha)\rtimes_{\mu}G$ and the quotient
\begin{equation}\label{eq-mures}
A\rtimes_{\mu|_H}H:=(A\rtimes_{\max}H)/I_{\mu|_H}.
\end{equation}

The following definition is taken from  \cite{Buss:2015ty}:
\begin{definition}\label{def-mures}
Let $\rtimes_{\mu}$ be a crossed-product functor for $G$.
Then the assignment $(A,\alpha)\mapsto A\rtimes_{\mu|_H}H$ with $ A\rtimes_{\mu|_H}H$ constructed as
above is called the {\em restriction} of $\rtimes_\mu$ to $H$.
\end{definition}

In \cite[Proposition 6.6]{Buss:2015ty} we also observed that 
for a second countable locally compact group $G$, the Baum-Connes assembly map
$$\as^{\mu|_H}_{(A,H)}: K_*^{\top}(H,A)\to K_*(A\rtimes_{\mu|_H}H)$$
is an isomorphism if and only if the assembly map
$$\as^{\mu}_{(\Ind_H^GA,G)}: K_*^{\top}(G, \Ind_H^GA)\to K_*(\Ind_H^GA\rtimes_{\Ind\alpha,\mu}G)$$
is an isomorphism. In particular, if $G$ satisfies the analogue of the Baum-Connes conjecture
 for the $\rtimes_{\mu}$-crossed product, then $H$ satisfies
the conjecture for the $\rtimes_{\mu|_H}$-crossed product.
Thus,  it is interesting to study the question whether
the restriction $\rtimes_{\E_G|_H}$ of the minimal exact crossed-product functor $\rtimes_{\E_G}$ for $G$ to a closed subgroup $H$ will
always be the minimal exact functor $\rtimes_{\E_H}$ for $H$, since this would imply
that the new conjecture of Baum, Guentner, and Willett (as explained in the introduction) 
passes to closed subgroups. In \cite[Theorem 7.13]{Buss:2015ty} we showed this 
for the case where $H$ is normal or cocompact in $G$.\footnote{In fact, we showed this for the smallest exact
{\em correspondence functor} $\rtimes_{\E_{\Cor}}$ which coincides with $\rtimes_\E$ on the category of separable $G$-algebras
if $G$ is second countable. But the same arguments as used in the proof of \cite[Theorem 7.13]{Buss:2015ty} apply directly to the 
smallest exact functors considered here.}
Below we shall give a proof of this fact 
if $H$ is open in $G$. We need the following result, which follows by the same arguments as 
used in the proof of  \cite[Lemma 7.6]{Buss:2015ty}.  In what follows, we denote by $i_B:B\to M(B\rtimes_\mu G)$ and $i_G:G\to M(B\rtimes_\mu G)$ the canonical embeddings into the multiplier algebra of a crossed product $B\rtimes_\mu G$.

\begin{lemma}\label{subgroups}
Let $H$ be a closed subgroup of $G$, and let $B$ be a $G$-algebra.
Then the canonical mapping
$i_B\rtimes i_G|_H\colon B\rtimes_{\max} H\to \M(B\rtimes_{\max}G)$ factors to a well-defined \Star{}homomorphism
$$i_B^{\E_G}\rtimes i_G^{\E_G}|_H\colon B\rtimes_{\E_H} H\to \M(B\rtimes_{\E_G}G). \eqno \qed$$
\end{lemma}

Note  that if $H$ is open in $G$, the above homomorphism extends the canonical inclusion 
of $C_c(H, B)$ into $C_c(G,B)$, and hence it takes its image in $B\rtimes_{\E_G}G$.
We are now ready for

\begin{theorem}\label{main}
Suppose that $H$ is an open subgroup of the locally compact group $G$. Then 
$\rtimes_{\E_G|_H}=\rtimes_{\E_H}$.
\end{theorem}
\begin{proof} For the proof we shall use a special form of Green's imprimitivity theorem in case where $H$ is open in $G$.
For this let $\chi_{eH}$ denote the characteristic function of the coset $\{eH\}\subseteq G/H$, viewed as a projection 
in the multiplier algebra $\M(\Ind_H^GA)$ in the canonical way, and let $p\in \M(\Ind_H^GA\rtimes_{\max}G)$ denote 
its image in the crossed product. Then it is shown in \cite[Proposition 2.6.8]{Echterhoff:2009jo}
that $p$ is a full projection such that $p(\Ind_H^GA\rtimes_{\max}G)p=A\rtimes_\max H$ and 
the resulting $\Ind_H^GA\rtimes_{\max}G$-$A\rtimes_{\max}H$ equivalence bimodule $(\Ind_H^GA\rtimes_{\max}G)p$ 
is isomorphic to Green's equivalence bimodule $X(A,\alpha)$. 
This implies that the image of $p$ in $\M(\Ind_H^GA\rtimes_{\E_G}G)$ (which we also denote by $p$) is a full projection such that 
$A\rtimes_{\E_G|_H}G=p(\Ind_H^GA\rtimes_{\E_G}G)p$.
It follows from \cite[Theorem 6.3]{Buss:2015ty} that $\rtimes_{\E_G|_H}$ is an exact functor for $H$. Thus, by minimality of $\E_H$, the identity 
on $C_c(G,A)$ induces a quotient map $q: A\rtimes_{\E_G|_H}H\to A\rtimes_{\E_H}H$. 

We now construct an inverse for $q$. Indeed, by Lemma \ref{subgroups} we have a canonical $*$-homomorphism 
$\varphi: \Ind_H^GA\rtimes_{\E_H}H\to \Ind_H^GA\rtimes_{\E_G}G$ which extends the canonical inclusion of 
$C_c(H,\Ind_H^GA)$ into $C_c(G, \Ind_H^GA)$.  As an $H$-algebra, $\Ind_H^GA$ decomposes as a direct sum 
$A\oplus I$, where we identify $A$ with the functions in $\Ind_H^GA$ which live on the coset $eH$ and $I$ with the functions 
which vanish on $eH$. This implies a decomposition $\Ind_H^GA\rtimes_{\E_H}H\cong (A\rtimes_{\E_H}H)\oplus (I\rtimes_{\E_H}H)$
and it is easily verified on functions in $C_c(H,A)$ that the homomorphism $\varphi$ maps the summand 
$A\rtimes_{\E_H}H$ into the corner $p(\Ind_H^GA\rtimes_{\E_G}G)p\cong A\rtimes_{\E_G|_H}H$, thus
 providing an inverse for $q$.
\end{proof}

\begin{remark} Using \cite[Lemma 7.6]{Buss:2015ty} instead of Lemma \ref{subgroups} in the above proof, the same arguments 
as used above also show that the restriction of the minimal exact correspondence functor $\rtimes_{\E_{\Cor}^G}$ of $G$
to an open subgroup $H$ of $G$ coincides with the minimal exact correspondence functor $\rtimes_{\E_{\Cor}^H}$ of $H$.
\end{remark}

\section{Some Questions}\label{sec questions}
There are still many important open questions about the smallest exact crossed-product functor $\rtimes_\E$. Here are some of them:

\begin{question}
 Is it true that the smallest exact crossed-product functor for $G$ is automatically a correspondence functor?
 \medskip
 
Since exact functors automatically satisfy the ideal property, it follows from \cite[Theorem 4.9]{Buss:2014aa} that  being a correspondence functor  is equivalent to any of the following 
 assertions:
 \begin{enumerate}
 \item For each $G$-algebra $A$ and $G$-equivariant projection $p\in \M(A)$, the 
 canonical map $pAp\rtimes_\E G\to A\rtimes_\E G$ is injective.
 \item For each $G$-algebra $A$ and $G$-equivariant {\em full} projection $p\in \M(A)$, the 
 canonical map $pAp\rtimes_\E G\to A\rtimes_\E G$ is injective.
 \item For each (full) $G$-invariant hereditary subalgebra $B$ of a $G$-algebra $A$, the 
 canonical map $B\rtimes_\E G\to A\rtimes_\E G$ is injective.
\item $\rtimes_\E$ is {\em strongly Morita compatible} in the sense that for any $G$-equivariant Morita equivalence bimodule
$(X,\gamma)$  between two $G$-algebras $A$ and $B$, the canonical $C_c(G,A)$-$C_c(G,B)$ bimodule
$C_c(G,X)$ completes to give an $A\rtimes_\E G$-$B\rtimes_\E G$ equivalence bimodule. 
\item The functor $\rtimes_\E$ is functorial for $G$-equivariant ccp maps.
\end{enumerate}
Note that it is shown in Corollary \ref{str mor com} above that all this holds on the category of $\sigma$-unital $G$-algebras.
\end{question}

\begin{question} Is the smallest exact crossed-product functor $\rtimes_\E$ injective?
\medskip

We say that a crossed-product functor $\rtimes_\mu$ is {\em injective}, if for every injective $G$-equivariant $*$-homomorphism
$\varphi:A\to B$ the descent $\varphi\rtimes_\mu G: B\rtimes_\mu G\to A\rtimes_\mu G$ is injective as well.
The reduced crossed-product functor $\rtimes_r$ is well known to be injective but lacks exactness in general.
At some early point of the project we thought we could show that injectivity holds for $\rtimes_\E$, but our argument 
had a serious gap. We then thought we had an argument proving that $\rtimes_\E$ is not injective in general, but again found
a gap in the proof. So right now, we have no clue about the correct answer  to this question.
Indeed, in the moment we do not know of any example of an exact and injective crossed-product functor 
for a non-exact group. 

Note that injectivity of $\rtimes_\E$, if true,  would imply some nice properties of this functor: It would be continuous for 
general inductive limits of $G$-algebras, i.e., we would get $(\lim_i A_i)\rtimes_\E G=\lim_i(A_i\rtimes_\E G)$ for every 
directed system of $G$-algebras $(A_i, \varphi_i)_{i\in I}$. It would also imply that the functor $\rtimes_\E$ 
preserves continuous fields of $C^*$-algebras in the sense that if $G$ acts fibrewise  on the section algebra 
$A$ of a continuous field of $C^*$-algebras over a base space $X$ with fibres $A_x$, $x\in X$, then 
$A\rtimes_\E G$ would be the section algebra of a continuous field of $C^*$-algebras with fibres $A_x\rtimes_\E G$.
\end{question}

\begin{question} Suppose that $G=N\times H$ is the product of two groups. Can we decompose the 
crossed product $A\rtimes_{\E_G}G=A\rtimes_{\E_G}(N\times H)$ as an iterated crossed product
$(A\rtimes_{\E_N}N)\rtimes_{\E_H}H$? Does it hold for discrete groups $N$ and $H$?
\medskip 

A positive answer would give the first step for proving that for general closed normal subgroups $N\subseteq G$
we could write $A\rtimes_{\E_G}G$ as an iterated crossed product $(A\rtimes_\E N)\rtimes_{\E_{G/N}}G/N$,
where in general the outer crossed product has to be viewed as a  twisted crossed product compatible 
with the functor $\rtimes_{\E_{G/N}}$ for the quotient group $G/N$. Such a decomposition would give 
a major step for a proof  that the new formulation of the Baum-Connes conjecture enjoys the same 
permanence properties as were shown for the classical conjecture in \cite{Chabert:2001hl} and \cite{Chabert:2004fj}.
We refer to \cite[Section 8]{Buss:2015ty}  for a discussion of this problem.
\end{question}

\begin{question}
Let $H$ be a closed subgroup for $G$. Can we always show that the smallest exact crossed-product functor $\rtimes_{\E_G}$ restricts
to the minimal exact crossed-product functor $\rtimes_{\E_H}$ for $H$? 
\medskip 

So far, we only know this if $H$ is open in $G$ (by Section \ref{sec-BC}) above, and for normal and cocompact subgroups $H$ of $G$
(by \cite[Theorem 7.13]{Buss:2015ty}). A positive answer would imply that the validity of the new formulation  of the
Baum-Connes conjecture for a group $G$ would pass to all closed subgroups of $G$.
\end{question}




\maketitle
\newpage
\appendix
\section{Erratum}
Unfortunately, there are two severe gaps in the main body of this paper.   We are very grateful to Yosuke Kubota who pointed these out  to us.  The first of these can be fixed under a mild additional hypothesis; we sketch how to do this below.  The second seems much more serious, and we are currently unable to fix it.  

There is also a third mistake in a side-remark that we record below; this third error does not affect the rest of the paper at all.  

The first problem appears in the proof of Proposition \ref{norm real}.  The first part of the proof of this proposition, dealing with the general case, seems correct. The problem arises with the additional argument dealing with the unital case.  Specifically, when trying to prove 
exactness of the left column of the diagram (5) on page 2051, we construct a map $E: \tilde{C}\rtimes_\mu G\to \tilde{C}\rtimes_\mu G$
and claim at the bottom of that page that it maps $C_c(G,J)$ into $C_c(G,I)$.  As pointed out to us by Yosuke Kubota, this is not true in general. 
However, under the mild extra assumption that the crossed-product functor $\rtimes_\mu$ satisfies the
ideal property the proof of Proposition 2.4 can be fixed; in particular, this applies when $\rtimes_\mu$ is the reduced crossed product, which is the most important special instance of Proposition 2.4.  The ideal property means that for every $G$-invariant ideal $I\subseteq A$ the crossed product $I\rtimes_\mu G$ injects into $A\rtimes_\mu G$. The following lemma is due to Narutaka Ozawa and  was communicated to us by Yosuke Kubota.

\begin{lemma}\label{lem-Oz}
Suppose that $\rtimes_\mu$ has  the ideal property.
Then every  short exact sequence of $G$-algebras   $0\to I\to A\to \C\to 0$
descends to a short exact sequence
$$0\to  I\rtimes_\mu G\to A\rtimes_\mu G\to \C\rtimes_\mu G\to 0.$$
\end{lemma}
\begin{proof}
Consider the generalized  homomorphism $\C\to M(A); \lambda \mapsto \lambda 1_{M(A)}$. By the ideal property, this  
descends to a nondegenerate morphism $\C\rtimes_\mu G\to M(A\rtimes_\mu G)$ (see \cite[Lemma 3.3]{Buss:2014aa})
and therefore  uniquely extends 
to a $*$-homomorphism $\phi: M(\C\rtimes_\mu G)\to M(A\rtimes_\mu G)$ which splits the extension 
$M(A\rtimes_\mu G)\to M(\C\rtimes_\mu G)$ of the 
quotient map $q: A\rtimes_\mu G\to \C\rtimes_\mu G$. 
Assume now that $x\in A\rtimes_\mu G$ is mapped to $0$ in $\C\rtimes_\mu G$ and let $(x_n)_n$ be a sequence
in $C_c(G,A)$ which converges to $x$ in norm. Let $(e_i)_{i\in I}$ be a bounded approximate unit  for 
$A\rtimes_\mu G$ which lies in $C_c(G,A)$.
Then $x_{n,i}:=\big(x_n-q\circ \phi(x_n)\big) e_i$ is an element of
 $C_c(G,A)$ such that $q(x_{n,i})=0$, and hence $x_{n,i}\in C_c(G,I)$  for all $(n,i)\in \N\times I$.
 On the other hand we  have $x_{n,i}\to x$ in norm,  hence  $x\in I\rtimes_\mu G$,
 which, by the ideal property,  coincides with the norm-closure of  $C_c(G,I)$ in $A\rtimes_\mu G$.
 \end{proof}

Now, if $\rtimes_\mu$ satisfies the ideal property,  the above lemma implies that the two left columns of diagram (5)  on p.~ 10
 are exact.
The proof of Proposition \ref{norm real} then follows as given.  As a result of this, Proposition \ref{group alg} should also assume the ideal property in its statement.  No other results in Sections \ref{sec-halfexact} or \ref{ex sec} seem to be affected by this mistake.  From Section \ref{sec-lift}, the following results are affected: the parts of Proposition \ref{ep ind}, Proposition \ref{welp ind}, and Corollary \ref{lp=red} dealing with the unital case.  All {of them}
hold if we assume the functor used has the ideal property.  No results from Section \ref{sec-BC} appear to be affected. \\

The second, and most severe, problem in the paper comes from Lemma \ref{always ext}, which was an important step in our proof of Proposition \ref{q mor com}.  Proposition \ref{q mor com} says that the smallest exact crossed-product  
$\rtimes_{\epsilon(\mu)}$ 
which dominates a Morita compatible functor $\rtimes_\mu$ with the ideal property is automatically Morita compatible.  We do not know if Proposition \ref{q mor com} is correct or not, but Lemma \ref{always ext} is definitely false.  The following counterexample to Lemma \ref{always ext}
resulted from a discussion with Timo Siebenand.

\begin{example}  Following a construction of Brown and Guentner, we construct a crossed-product functor for a group $G$
as the completion of $C_c(G,A)$ by the norm
$$\|f\|_\mu:=\sup\{\|\pi\rtimes U(f)\|: (\pi,U) \;\text{covariant rep. of $(A,G,\alpha)$ s.t. $U\prec \lambda_G$}\},$$
where ``$\prec$'' denotes weak containment of representations. 
Now, up to unitary equivalence, every nondegenerate covariant representation  $(\pi, U)$  of 
$(\K(L^2(G)), G, \Ad\lambda_G)$ is given as $(\pi, U)=(\id_{\K(L^2(G))}\otimes 1_{\h}, \lambda_G\otimes V)$  acting on  $L^2(G)\otimes \h$ for some Hilbert space $\h$.
Therefore $U\prec \lambda_G$ by Fell's trick. 
It follows that 
$$\K(L^2(G))\rtimes_{\Ad\lambda_G, \mu}G=\K(L^2(G))\rtimes_{\Ad\lambda_G, \max}G\cong \K(L^2(G))\otimes C^*_\max(G),$$
while  $\K(L^2(G))\otimes (\C\rtimes_\mu G)= \K(L^2(G))\otimes  C_r^*(G)$. 
Since $\rtimes_\mu$ satisfies the ideal property,  this example contradicts Lemma \ref{always ext} whenever $G$ is not amenable.
\end{example}

A version of Lemma 4.3 could still be true under the extra assumption that $\rtimes_\mu$ is Morita compatible itself.
Note that the functor $\rtimes_\mu$ of the above example is not. This would save Proposition \ref{q mor com}
and all its consequences. Unfortunately, so far  we were not able to give a proof of the lemma even under this extra assumption.
Therefore Proposition \ref{q mor com} together with Corollaries \ref{str mor com}, \ref{min ex mc}, \ref{corr fun}, and \ref{4.8}
are still open.  
In particular it is still open, whether the group algebra $C^*_{\epsilon_M}(G)=\C\rtimes_{\epsilon_M}G$ for the  
smallest {\em Morita compatible} exact crossed-product functor $\rtimes_{\epsilon_M}$ coincides with the 
reduced group algebra $C_r^*(G)$. 

The results in Section \ref{sec-lift} and Section  \ref{sec-BC}  appear unaffected by these problems. In particular, Theorem \ref{main} holds true 
with $\rtimes_{\epsilon_G}$ and $\rtimes_{\epsilon_H}$ denoting either the
smallest exact crossed-product functors or the smallest exact Morita compatible exact crossed-product functors for $G$ and $H$.\\

We finally would like to point out a third mistake, which appears in the paragraph before Theorem \ref{thm-exact}.
There we claim that for a $G$-algebra $A$ the canonical map 
$$\iota^{**}:A^{**}\to (A\rtimes_\mu G)^{**}$$
is always injective, where $A^{**}$ denotes  the double dual of $A$.
This is indeed true if $G$ is discrete, but it does  not hold in general. 
We did not use this statement anywhere in the paper.

To see a counterexample, let $G$ be any locally compact group acting on itself by translation.
Then $C_0(G)\rtimes G\cong \K(L^2(G))$, hence $(C_0(G)\rtimes G)^{**}=\B(L^2(G))$ and then  
$\iota^{**}$ maps $C_0(G)^{**}$ onto $C_0(G)''\cong L^{\infty}(G)\subseteq \B(L^2(G))$, where we represented 
$C_0(G)$ into $\B(L^2(G))$ via multiplication operators. But for general locally compact groups the 
map from $C_0(G)^{**}$ (which contains all characteristic functions $\chi_{\{g\}}$ of single points $g\in G$) to $L^{\infty}(G)$ is not injective.
{We should point out that if $\iota^{**}$ is not faithful on $A^{**}$, then it is also not 
faithful on the continuous part
$A^{**}_c=\{a\in A^{**}: g\mapsto \alpha^{**}_g(a)\;\text{norm continuous}\}$.
Indeed, since  $N:=\ker \iota^{**}$ is a von Neumann subalgebra of $A^{**}$ the unit $1_N$ of $N$ clearly lies in $A_c^{**}$.}

\section{The smallest exact Morita compatible crossed product}
In this section we want to show how the smallest exact Morita compatible crossed product $\rtimes_{\epsilon_M}$
for a locally compact group $G$,
which is used in the reformulation of the Baum-Connes conjecture by Baum, Guentner, and Willett in \cite{Baum:2013kx},
 can be constructed out of the smallest exact crossed product as described in this paper. 
Indeed, starting with any exact crossed product functor $\rtimes_\mu$ satisfying certain additional 
properties, we give a construction of a corresponding Morita compatible exact crossed product $\rtimes_{\mu_M}$
which will coincide with $\rtimes_\mu$ if this happened to be Morita compatible  already.
 If we  then start with the smallest exact crossed product $\rtimes_\eps$  
as constructed in this paper, we obtain the smallest Morita  compatible exact crossed product $\rtimes_{\eps_M}$.

Before we come to the construction, we first recall that for every unitary representation 
$u:G\to U(\h)$ of $G$ on some Hilbert space $\h$ we obtain canonical isomorphisms
\begin{equation}\label{Psi-max}
\Psi_{\max}^u: (A\rtimes_{\alpha,\max}G)\otimes \K(\h)\to (A\otimes\K(\h))\rtimes_{\alpha\otimes\Ad u, \max}G
\end{equation}
and 
\begin{equation}\label{Psi-red}
\Psi_{r}^u: (A\rtimes_{\alpha,r}G)\otimes \K(\h)\to (A\otimes\K(\h))\rtimes_{\alpha\otimes\Ad u, r}G
\end{equation}
which are both given on the level of continuous functions with compact supports by the map
$$\Psi_{alg}^u:  C_c(G,A)\odot \mathcal K\to C_c(G, A\otimes \K), \Psi_{alg}(f\otimes k)(g)=f(g)(1\otimes k u_g^*).$$

\begin{definition}\label{def-Morita}
We say that the crossed  product functor $\rtimes_\mu$  is {\em Morita compatible} if for any unitary representation 
$u:G\to U(\h)$ the dashed arrow in the diagram below can be filled in with a $*$-homomorphism $\Psi_{\mu}^u$ 
$$
\xymatrix{ (A\rtimes_{\alpha,\max}G)\otimes \K(\h) \ar[d] \ar[r]^-{\Psi_{\max}^u} & (A\otimes\K(\h))\rtimes_{\alpha\otimes\Ad u, \max}G \ar[d] \\
(A\rtimes_{\alpha,\mu}G)\otimes \K(\h) \ar@{-->}[r]^-{\Psi_{\mu}^u} &  (A\otimes\K(\h))\rtimes_{\alpha\otimes\Ad u, \mu}G }
$$
(here the vertical arrows are the canonical quotients).
\end{definition}

Now let $\rtimes_\mu$  be any crossed-product functor for a given locally compact group $G$
and let  $\alpha:G\to \Aut(A)$ be an action of $G$  on the $C^*$-algebra $A$. 
Let $\K=\K(L^2(G))$ and let
$$\Psi_\max^\lambda: (A\rtimes_{\alpha, \max}G)\otimes \mathcal K\to \big(A\otimes\K\big)\rtimes_{\alpha\otimes\Ad\lambda, \max}G$$
be as above, where $\lambda:G\to U(L^2(G))$ denotes the left regular representation of $G$. Let
$$Q_\mu: \big(A\otimes\K\big)\rtimes_{\alpha\otimes\Ad\lambda, \max}G\to \big(A\otimes\K\big)\rtimes_{\alpha\otimes\Ad\lambda, \mu}G$$
denote the quotient map and let 
$$j_{A\rtimes G}:A\rtimes_{\max}G\to \M(A\rtimes_{\max}G\otimes \K); j_{A\rtimes G}(x)=x\otimes 1$$
denote the canonical map. For $f\in C_c(G,A)$ we define
$$\|f\|_{\mu_M}:=\|Q_\mu\circ \Psi_{\max}^{\lambda}\circ j_{A\rtimes G}(f)\|.$$
The crossed product 
$ A\rtimes_{\alpha,\mu_M}G$ is then defined as the completion of $C_c(G,A)$ with respect to $\|\cdot\|_{\mu_M}$.
Of course, by construction, the $*$-homomorphism 
$$Q_\mu\circ \Psi\circ j_{A\rtimes G}:A\rtimes_{\max}G\to \M\big(\big(A\otimes\K\big)\rtimes_{\alpha\otimes\Ad\lambda, \mu}G\big)$$
factors through an embedding, say 
\begin{equation}\label{map}
\Phi_\mu^A: A\rtimes_{\alpha,\mu_M}G\to \M\big(\big(A\otimes\K\big)\rtimes_{\alpha\otimes\Ad\lambda, \mu}G\big).
\end{equation}

Let us have a slightly different look at our construction which might help to get a better feeling for it. For this recall that
any quotient $C$ of a tensor product $D\otimes \K$ for some $C^*$-algebra $D$ must be of the form $D/J\otimes \K$. 
The ideal $J$ is then given as the kernel of $q\circ j_D:D\to \M(C)$, where $j_D:D\to \M(D\otimes \K)$ is the canonical inclusion and
$q:D\otimes  \K\to C$ denotes the quotient map.

Thus, if we consider  the composition 
$$Q_\mu\circ \Psi_{\max}^\lambda: (A\rtimes_{\max}G)\otimes \K\to (A\otimes\K)\otimes_{\alpha\otimes\Ad\lambda, \mu}G$$
let $J:=\ker(Q_\mu\circ \Psi)\subseteq  (A\rtimes_{\max}G)\otimes \K$. It then follows from our construction that
$$
 \big( (A\rtimes_{\max}G)\otimes \K\big)/J\cong (A\rtimes_{\alpha,\mu_M}G)\otimes \K.
$$
and we therefore get an isomorphism
\begin{equation}\label{eq-muM}
(A\otimes\K)\rtimes_{\alpha\otimes\lambda,\mu}G \cong (A\rtimes_{\alpha,\mu_M}G)\otimes \K.
 \end{equation}
 which is natural in $(A,\alpha)$.

\begin{lemma}\label{lem-functor}
$(A,\alpha)\mapsto A\rtimes_{\alpha,\mu_M}G$ is a crossed-product functor for $G$ and if $\rtimes_\mu$  is exact, then so is 
$\rtimes_{\mu_M}$.
\end{lemma}
\begin{proof} Let  $\Phi:(A,\alpha)\to (B,\beta)$ be a $G$-equivariant $*$-homomorphism. We then obtain a $G$-equivariant $*$-homomorphism
$\Phi\otimes\id_{\K}: A\otimes \K\to B\otimes  \K$ which then descends to a $*$-homomorphism
$$(\Phi\otimes\K)\rtimes_\mu  G: (A\otimes  \K)\rtimes_{\alpha\otimes\Ad\lambda, \mu}G\to (B\otimes  \K)\rtimes_{\beta\otimes\Ad\lambda, \mu}G.$$
It is then easy to check that  we have a commutative diagram of $*$-homomorphisms
$$
\begin{CD} A\rtimes_{\max}G  @>\Phi\rtimes G>> B\rtimes_{\max} G\\
@V \Phi_\mu^AVV  @VV\Phi_\mu^BV\\
(A\otimes  \K)\rtimes_{\alpha\otimes\Ad\lambda, \mu}G @>>(\Phi\otimes\K)\rtimes_\mu  G> (B\otimes  \K)\rtimes_{\beta\otimes\Ad\lambda, \mu}G
\end{CD}
$$
with $\Phi_\mu^A, \Phi_\mu^B$ as in (\ref{map}) viewed as maps on the maximal crossed products.
It follows that every element in the kernel of $\Phi_\mu^A$ is mapped into the kernel of $\Phi_\mu^B$, and therefore
the homomorphism $\Phi\rtimes G:A\rtimes_{\max}G\to B\rtimes_{\max}G$ factors through a $*$-homomorphism 
$\Phi\rtimes_{\mu_M} G:A\rtimes_{\mu_M}G\to B\rtimes_{\mu_M}G$ which is given on the level of $C_c(G,A)$ by $f\mapsto \Phi\circ f$.
This proves that $\rtimes_{\mu_M}$ is a crossed-product functor.

Assume  now that $\rtimes_\mu$ is exact and let
$0\to I\to A\to A/I\to 0$ be a short  exact sequence of $G$-algebras. By nuclearity of  $\K$ and exactness of $\rtimes_\mu$ we get a short exact sequence
$$0\to (I\otimes\K)\rtimes_\mu G\to (A\otimes\K)\rtimes_\mu G\to (A/I\otimes\K)\rtimes_\mu G\to 0$$
which by (\ref{eq-muM}) translates  to the short  exact sequence
$$  0\to (I\rtimes_{\mu_M}G)\otimes \K \to (A\rtimes_{\mu_M}G)\otimes \K \to (A/I\rtimes_{\mu_M}G)\otimes \K \to 0.$$
But this implies  that the sequence
$$0\to I\rtimes_{\mu_M}G \to A\rtimes_{\mu_M}G \to A/I\rtimes_{\mu_M}G$$
is exact as well.
\end{proof}

\begin{proposition}\label{prop-Morita}
Suppose that $\rtimes_\mu$  is a  crossed-product functor which satisfies the following  conditions:
 For each $G$-algebra $A$ and any algebra $\K=\K(\h)$ of compact operators on some Hilbert space $\h$  the  
map 
$$\Psi_{alg}^{1_\h}: C_c(G,A)\odot \K\to C_c(G, A\otimes \K); \Psi_{alg}(f\otimes k)(g)=f(g)\otimes k$$
extends to an isomorphism 
\begin{equation}\label{Psi-id}
(A\rtimes_{\alpha,\mu} G)\otimes \K\cong (A\otimes\K)\rtimes_{\alpha\otimes\id_{\K},\mu}G.
\end{equation}
Then the functor $\rtimes_{\mu_M}$ constructed above is Morita compatible.
\end{proposition}
\begin{proof} 
We need to show that for any unitary representation $u:G\to U(\h)$ the $*$-isomorphism 
$$
\Psi_{\max}^u: (A\rtimes_{\alpha,\max}G)\otimes \K(\h)\to (A\otimes\K(\h))\rtimes_{\alpha\otimes\Ad u, \max}G
$$
of (\ref{Psi-max}) factors through an isomorphism
$$
\Psi_{\mu_M}^u: (A\rtimes_{\alpha,\mu_M}G)\otimes \K(\h)\to (A\otimes\K(\h))\rtimes_{\alpha\otimes\Ad u, \mu_M}G.
$$
In order to see this we first  recall the isomorphism
$$
A\rtimes_{\alpha,\mu_M}G\otimes \K(L^2(G))\cong (A\otimes\K(L^2(G)))\rtimes_{\alpha\otimes\Ad \lambda,\mu}G.
$$
which is given on the level of functions by $\Psi_{alg}^\lambda$. 
Composing this with the isomorphism in (\ref{Psi-id}) with $A\rtimes_{\alpha,\mu} G$ replaced by 
$(A\otimes\K(L^2(G)))\rtimes_{\alpha\otimes\Ad \lambda,\mu}G$ we obtain an isomorphism 
\begin{equation}\label{iso1}
\begin{split}
(A\rtimes_{\alpha,\mu_M}G)\otimes  &\K(L^2(G))\otimes \K(\h)\\
&{\cong} \big(A\otimes\K(L^2(G))\otimes \K(\h)\big)\rtimes_{\alpha\otimes\Ad (\lambda\otimes 1_h),\mu}G
\end{split}
\end{equation}
which is given on the level of functions by the map
$$\Psi_{alg}^{\lambda\otimes 1_\h}: C_c(G,A)\otimes \K(L^2(G)\otimes \h)\to C_c(G, A\otimes\K(L^2(G)\otimes \h)).$$
Let $W\in U(L^2(G)\otimes \h)$ be the unitary given by $W(\xi)(g)=u_g\xi(g)$ for $\xi\in L^2(G,\h)\cong L^2(G)\otimes\h$.
Then a short computation shows that 
$$W(\lambda_g\otimes 1_\h)W^*=\lambda_g\otimes u_g$$
for all $g\in G$ (this is known as Fell's trick).  We then get the following commutative diagram of maps
$$
\begin{CD}
C_c(G,A)\otimes \K(L^2(G)\otimes \h) @>\Psi_{alg}^{\lambda\otimes 1_\h}>> C_0(G, A\otimes\K(L^2(G)\otimes \h))\\
@V\id\otimes \Ad W VV   @V \Ad W\rtimes_{alg}G VV\\
C_c(G,A)\otimes \K(L^2(G)\otimes \h) @>\Psi_{alg}^{\lambda\otimes u}>> C_0(G, A\otimes\K(L^2(G)\otimes \h))
\end{CD}
$$
which  combines with (\ref{iso1}) to show that the bottom arrow extends to an isomorphism
\begin{equation}\label{iso2}
\begin{split}
(A\rtimes_{\alpha,\mu_M}G)\otimes  &\K(L^2(G))\otimes \K(\h)\\
&{\cong} \big(A\otimes\K(L^2(G))\otimes \K(\h)\big)\rtimes_{\alpha\otimes\Ad (\lambda\otimes u),\mu}G
\end{split}
\end{equation}
Applying the flip $\Sigma:\K(L^2(G))\otimes \K(\h)\to \K(\h)\otimes \K(L^2(G))$ together with the isomorphism
(\ref{eq-muM}) applied to the $G$-algebra $(A\otimes \K(\h), \alpha\otimes \Ad u)$ then gives 
a chain of isomorphisms
\begin{align*}
(A\rtimes_{\alpha,\mu_M}G)&\otimes\K(\h)\otimes\K(L^2(G))\\
&\stackrel{\id\otimes \Sigma}\cong
(A\rtimes_{\alpha,\mu_M}G)\otimes\K(L^2(G))\otimes\K(L^2(\h))\\
&\stackrel{(\ref{iso2})}{\cong} (A\otimes\K(L^2(G))\otimes \K(\h))\rtimes_{\alpha\otimes\Ad(\lambda\otimes u),\mu}G\\
&\stackrel{\Sigma\rtimes G}{\cong} (A\otimes\K(\h)\otimes \K(L^2(G)))\rtimes_{\alpha\otimes\Ad(u\otimes \lambda),\mu}G\\
&\stackrel{(\ref{eq-muM})}{\cong} \big((A\otimes\K(\h))\rtimes_{\alpha\otimes  \Ad u, \mu_M}G\big)\otimes \K(L^2(G)),
\end{align*}
which extends the map $\Psi_{\alg}^u\otimes\id_{\K(\h)}$.
It follows that 
$$\Psi_{alg}^u:C_c(G,A)\otimes \K(\h)\to C_c(G, A\otimes \K(\h))$$ extends to an isomorphism
$$(A\rtimes_{\alpha,\mu_M}G)\otimes\K(\h)\cong (A\otimes\K(\h))\rtimes_{\alpha\otimes  \Ad u, \mu_M}G$$
and the result follows.
\end{proof}

\begin{lemma}\label{lem-exact-Morita}
The smallest exact crossed-product functor $\rtimes_\epsilon$ satisfies condition (\ref{Psi-id}) of 
Proposition \ref{prop-Morita}
\end{lemma}
\begin{proof}
We need to show that the map 
$$\Psi_{alg}: C_c(G,A)\odot \K\to C_c(G, A\otimes \K); \Psi_{alg}(f\otimes k)(g)=f(g)\otimes k$$
extends to an isomorphism
$$\Psi_{\eps}:(A\rtimes_{\alpha,\epsilon} G)\otimes \K\stackrel{\cong}{\to} (A\otimes\K)\rtimes_{\alpha\otimes\id_{\K},\epsilon}G.$$
The proof of Proposition \ref{q mor com} shows that the inverse 
$$\Phi_{\max}:(A\otimes\K)\rtimes_{\alpha\otimes\id_{\K},\max}G\to (A\rtimes_{\alpha,\max} G)\otimes \K$$
of $\Psi_{\max}$ factors through a surjective $*$-homomorphism
$$\Phi_{\eps}:(A\otimes\K)\rtimes_{\alpha\otimes\id_{\K},\eps}G\to (A\rtimes_{\alpha,\eps} G)\otimes \K.$$
Thus it suffices to show that $\Psi_{\max}$ factors through a $*$-homomorphism. 
But this follows from the proof of Lemma \ref{always ext}, which still holds in the case where $u=1_{\h}$.
\end{proof}

Thus Proposition \ref{prop-Morita} can be applied to the smallest exact crossed product $\rtimes_\epsilon$. We then get

\begin{corollary}\label{cor-exact-Morita}
Let $\rtimes_\eps$ denote the smallest exact crossed product  functor for the group $G$.
Then $\rtimes_{\eps_M}$ is the smallest exact Morita compatible crossed product  functor for $G$.
\end{corollary}
\begin{proof} All that remains to be checked is that $\rtimes_{\eps_M}$ is dominated by any other 
exact Morita compatible crossed-product functor $\rtimes_\mu$. For this consider the commutative diagram
$$
\begin{CD}
C_c(G,A)\otimes \K(L^2(G)) @>\Psi_{alg}^\lambda>\cong > (A\otimes \K(L^2(G)))\rtimes_{\alpha\otimes \Ad\lambda, \mu}G\\
@V\id VV    @VV q_\epsilon V\\
C_c(G,A) \otimes \K(L^2(G)) @>\Psi_{alg}^\lambda>\cong > (A\otimes \K(L^2(G)))\rtimes_{\alpha\otimes \Ad\lambda, \epsilon}G
\end{CD}
$$
where the right vertical arrow exists since $\rtimes_\mu$  dominates $\rtimes_\epsilon$. 
Since $\Psi_{alg}^\lambda$ extends to an  isomorphism
$$(A\rtimes_{\alpha,\mu}G)\otimes \K(L^2(G)) \cong (A\otimes \K(L^2(G)))\rtimes_{\alpha\otimes \Ad\lambda, \mu}G$$
in the top line and to an isomorphism 
$$(A\rtimes_{\alpha,\epsilon_M}G)\otimes \K(L^2(G)) \cong (A\otimes \K(L^2(G)))\rtimes_{\alpha\otimes \Ad\lambda, \epsilon}G$$
in the bottom line of the diagram, it follows that the left vertical arrow extends to a quotient map
$$q_{\epsilon_M}\otimes \id_{\K}:(A\rtimes_{\alpha,\mu}G)\otimes \K(L^2(G))  \to (A\rtimes_{\alpha,\epsilon_M}G)\otimes \K(L^2(G))$$
and the result follows.
\end{proof}

\noindent
{\bf Final discussion:} If we want to show that the group algebra of the smallest exact Morita compatible 
crossed-product functor $\rtimes_{\eps_M}$ coincides with the reduced group algebra $C_r^*(G)$, we 
 need to show that we have a canonical isomorphism
$$C_\epsilon^*(G)\otimes \K(L^2(G))\cong \K(L^2(G))\rtimes_{\Ad\lambda,\eps}G.$$
By the above results this would imply that $C_{\epsilon_M}^*(G)=C_{\epsilon}^*(G)$, which by 
Proposition \ref{group alg}  is equal to $C_r^*(G)$. 
 
 More generally, if for every $G$-algebra $(A,\alpha)$  one could show that $\Psi_{alg}^\lambda$ extends to a 
$*$-homomorphism
 $$\Psi_{\epsilon}^\lambda: (A\rtimes_{\alpha,\epsilon}G)\otimes\K(L^2(G))\to (A\otimes \K(L^2(G)))\rtimes_{\alpha\otimes\Ad\lambda, \epsilon}G$$
then $\rtimes_\epsilon=\rtimes_{\epsilon_M}$ would be Morita compatible.


\end{document}